\documentclass{amsart}

\title{Based maps to Lagrangian Grassmannians, Quivers, and Bott Periodicity}
\date{Last compiled:  \today.  Last edited: July 9, 2026.}

\author{Jim Bryan}
\address{
Dept. of Math\\
University of British Columbia \\
Vancouver, British Columbia, Canada
}
\email{jbryan@math.ubc.ca}

\author{Ravi Vakil}
\address{
Dept. of Math\\
Stanford University\\
Palo Alto, California, USA
}
\email{rvakil@math.stanford.edu}

\usepackage[all]{xy}
\usepackage{diagbox}
\usepackage{float}
\usepackage{graphicx}
\usepackage{booktabs}
\usepackage{amsmath}
\usepackage{bm}
\usepackage{verbatim}
\usepackage{amsmath,amsthm,amsfonts,bbm}
\usepackage{amssymb}
\usepackage{times}
\usepackage{longtable}
\usepackage{caption}
\usepackage{array}
\usepackage{tikz-cd}
\usepackage{mathtools}
\usepackage{xcolor}
\usepackage{tikz}
\usetikzlibrary{arrows,positioning}

\newtheorem{theorem}{Theorem}[section]
\newtheorem{proposition}[theorem]{Proposition}

\newtheorem{lemma}[theorem]{Lemma}
\newtheorem{corollary}[theorem]{Corollary}

\theoremstyle{definition}

\newtheorem{remark}[theorem]{Remark}
\newtheorem{def-theorem}[theorem]{Definition-Theorem}
\newtheorem{definition}[theorem]{Definition}

\newtheorem{expectation}[theorem]{Expectation}

\usepackage{amsmath}
\usepackage{amsmath,amsthm,amsfonts}
\usepackage{times}

\newcommand{\CC} {{\mathbb C}}          
\newcommand{\ZZ} {{\mathbb Z}}		

\newcommand{\DD} {{\mathbb D}}
\newcommand{\JJ} {{\mathbb J}}
\newcommand{\kk} {{\mathbbm k}}
\renewcommand{\AA} {{\mathbb A}}

\newcommand{\HoG}{\textsf{HoG}}
\newcommand{\HoTop}{\textsf{HoTop}}

\newcommand{\PP}{\mathbb{P}}
\newcommand{\OO}{\mathcal{O}}

\newcommand{\Hom}{\operatorname{Hom}}
\newcommand{\Ker}{\operatorname{Ker}}

\newcommand{\Ext}{\operatorname{Ext}}

\newcommand{\Coker}{\operatorname{Coker}}
\newcommand{\coker}{\operatorname{coker}}
\newcommand{\im}{\operatorname{Im}}

\newcommand{\Sym}{\operatorname{Sym}}

\newcommand{\Isom}{\operatorname{Isom}}

\newcommand{\UL}[1]{\underline{#1}}
\newcommand{\alg}{\mathsf{alg}}
\newcommand{\open}{\mathsf{o}}
\newcommand{\stable}{\mathsf{s}}

\renewcommand{\top}{\mathsf{top}}
\newcommand{\Gr}{\operatorname{Gr}}
\newcommand{\LGr}{\operatorname{LGr}}
\newcommand{\OGr}{\operatorname{OGr}}
\newcommand{\LoopTwo}{\Omega^{2}_{d,\alg}}
\newcommand{\LoopTwoTop}{\Omega^{2}_{d,\top}}

\newcommand{\NoDLoopTwoTop}{\Omega^{2}_{\top}}
\newcommand{\Id}{\operatorname{Id}}

\newcommand{\homotopyeq}{\overset \sim \longleftrightarrow}
\newcommand{\Xpm}{X_{\pm,n}}
\newcommand{\Rpistar}{R\pi_{*}}

\begin{document}

\begin{abstract}
We give a quiver description of the space of based algebraic maps from
$\PP^{1}$ to the Lagrangian Grassmannian (and its orthogonal
counterpart). We show our descriptions lead to an algebro-geometric
refinement of some of the homotopy equivalences in real Bott
periodicity. In particular, we get an isomorphism in Larson and
Vakil's ``naive algebro-geometric homotopy category''  whose topological
realization (after specializing to $\CC$) recovers the classical homotopy equivalences
$\Omega^{2}(Sp/U)\homotopyeq BO\times \ZZ$ and
$\Omega^{2}(O/U)\homotopyeq BSp\times \ZZ$.

\end{abstract}

\maketitle 

\markboth{Genus Zero Maps to Lagrangian Grassmannians}  {Bryan-Vakil}


\section{Introduction}\label{sec: intro}

\subsection{Algebraic and topological double loop spaces and Bott periodicity.}

In this paper, we study the \emph{algebraic double loop space} of
Lagrangian Grassmannians and the orthogonal counterpart. We work over
a field $\kk $ with $\operatorname{char}(\kk)\neq 2$.   (With straighforward changes, we expect our arguments will work over $\operatorname{Spec} \ZZ [1/2]$, appropriately stated.)  We define the
algebraic double loop space as follows.

\begin{definition}\label{defn: Omega2alg(X)}
Let $(X,*)$ be a projective variety with a base (rational) point $*\in X$. Fix homogeneous
coordinates $(x_{0}:x_{1})$ on $\PP^{1}$ with point at infinity
$p_{\infty}=(1:0)$. Let $\LoopTwo(X)$ be the space of algebraic based
maps of degree $d \in \operatorname{CH}_1(X)$, that is
\[
\LoopTwo (X) = \left\{f:\PP^{1}\to X\text{ such that $f(p_{\infty})=*$
and $f_{*}([\PP^{1}])=d$ } \right\} .
\]
\end{definition}
This moduli space is naturally an algebraic variety.

In the analytic topology when $\kk =\CC$, $\CC \PP^{1}$ is homeomorphic to
$S^{2}$ and so there is a continuous inclusion of the algebraic double loop
space into the usual double loop space studied in topology:
\[
i:   \LoopTwo (X) \hookrightarrow \LoopTwoTop (X)
\]
where $\LoopTwoTop (X)$ is the space of continuous based maps of
degree $d\in H_{2}(X,\ZZ )$ (we will drop the $d$ from the notation
when $H_{2}(X,\ZZ )=0$). 

In the same spirit as using high degree polynomial functions to
approximate continuous functions, we may hope that for certain
varieties $X$ over $\CC$, the algebraic loop space approximates the
topological loop space in the sense that it captures most of its
homotopy theory.

\begin{expectation}
For certain varieties $X$ over $\CC$, 
\[
i:\LoopTwo (X) \hookrightarrow \LoopTwoTop
(X)
\]
is a \emph{homotopy approximation} meaning that 
\[
i_{*}:\pi_{k}(\LoopTwo (X)) \to \pi_{k}( \LoopTwoTop (X))
\]
is an isomorphism for all $k<B(d,X)$ and that $B\to \infty $ as $d\to
\infty$. 
\end{expectation}

In the 1980s and 1990s the above expectation was shown to hold for
generalized flag varieties in a series of papers by Segal \cite{Segal-1979}, Kirwan \cite{Kirwan-1986},
Mann-Milgram \cite{Mann-Milgram-93}, and Boyer-Mann-Hurtubise-Milgram \cite{Boyer-Mann-Hurtubise-Milgram}.

\begin{theorem}\label{thm: Mann-Milgram algebraic loop approximation
holds for generalized flag varieties}
Let $G$ be a reductive algebraic group, let $P\subset G$
be a parabolic subgroup, and let $X=G/P$ be the corresponding
generalized flag variety. Then the above expectation holds.
\end{theorem}

The (topological) double loop space $\LoopTwoTop (X)$ plays an
important role in homotopy theory, for example in Bott periodicity. To
state the version of Bott periodicity we want, we recall that $U(n)$,
$Sp(2n)$, and $O(2n)$ are the classical compact Lie groups of unitary,
symplectic, and orthogonal transformations. We also have the
infinite-dimensional versions of these groups (defined by direct limits of
inclusions) and their associated classifying spaces and homogeneous
spaces. For example:
\begin{align*}
U&=\lim_{n\to  \infty} U(n),\\
Sp/U& = \lim_{n\to  \infty} Sp(2n)/U(n),\\
BO& = \lim_{n\to  \infty} BO(n),\\
\end{align*}
and so forth.

One version of Bott periodicity is then given as follows.

\begin{theorem}[\cite{Bott-Periodicity-1959}]\label{thm: classical Bott periodicity as homotopy equivalences}
There exist the following homotopy equivalences:
\begin{align}
\LoopTwoTop (BU)&\homotopyeq BU \label{eqn: Omega2BU=BU}\\
\LoopTwoTop (Sp/U)&\homotopyeq BO  \label{eqn: Omega2Sp/U=BO}\\
\NoDLoopTwoTop (BO)&\homotopyeq O/U  \label{eqn: Omega2BO=O/U}\\
\LoopTwoTop (O/U)&\homotopyeq BSp  \label{eqn: Omega2O/U=BSp}\\
\NoDLoopTwoTop (BSp)&\homotopyeq Sp/U   \label{eqn: Omega2BSp=Sp/U}
\end{align}
\end{theorem}

Equation~\eqref{eqn: Omega2BU=BU} is referred to as \emph{complex Bott
periodicity} and it implies that the homotopy groups of $U$ are
2-periodic. Equations~\eqref{eqn: Omega2Sp/U=BO}--\eqref{eqn: Omega2BSp=Sp/U} 
are referred to as \emph{real Bott periodicity} and they imply for example that
the homotopy groups of $O$ are 8-periodic.

\subsection{Summary of main results}

The main results of this paper give explicit descriptions of $\LoopTwo
(X)$ when $X$ is a Grassmannian, a Lagrangian Grassmannian, or a maximal
isotropic orthogonal Grassmannian. Our description is in terms of
moduli spaces of quiver representations. We use our descriptions along
with Theorem~\ref{thm: Mann-Milgram algebraic loop approximation
holds for generalized flag varieties} to give new proofs of the
homotopy equivalences \eqref{eqn: Omega2BU=BU}, \eqref{eqn:
Omega2Sp/U=BO}, and \eqref{eqn: Omega2O/U=BSp} and so we may regard
our theorems as finite-dimensional, algebro-geometric refinements of Bott
periodicity. The study of the remaining equivalences in the Bott
periodicity theorem from the perspective of this paper will be pursued
in future work.

To state our results more explicitly, let $ \Gr(n,n+N) $ be the
Grassmannian of $n$-dimensional quotients of a fixed $(n+N)$-dimensional
vector space, let $ \LGr (n,2n) $ be the Lagrangian Grassmannian which
parameterizes $n$-dimensional Lagrangian quotients of a fixed
$2n$-dimensional symplectic vector space, and let $ \OGr (n,2n) $ be
the isotropic orthogonal Grassmannian which parameterizes
$n$-dimensional isotropic quotients of a fixed split $2n$-dimensional
orthogonal vector space. We note that over the complex numbers with
the analytic topology, these varieties are also given by homogeneous
spaces:
\begin{align*}
\Gr (n,n+N)&=U(n+N)/U(n)\times U(N),\\
\LGr (n,2n)&= Sp(2n)/U(n),\\
\OGr (n,2n)&= O(2n)/U(n) .
\end{align*}

Let $X$ be any of the above homogeneous spaces. Then $X$
is a projective variety (which is defined over any field) of
Picard number 1 and so the class of a based map $f:\PP^{1}\to X$ may
be regarded as a non-negative integer (the degree).  (For $\OGr
(n,2n)$ we take the connected component containing the chosen base
point; equivalently, the based mapping space only sees this component.)

\begin{theorem}\label{thm: quiver description of Loop2alg(Gr(n,n+N))}
Let $A$ be a vector space of dimension $d$ and let $U,W$ be vector
spaces of dimension $n$ and $N$, respectively. Let
\[
V=V_{n,N,d} =\Hom (A,A)\oplus \Hom (W,A)\oplus \Hom (A,U).
\]
Then $GL(A)$ acts naturally on $V$ and let $V^{\stable}\subset V$ be
the locus of points which are stable in the sense of Geometric
Invariant Theory (with respect to the trivial linearization). Then
there exists a natural isomorphism:
\[
\LoopTwo (\Gr (n,n+N))\cong V^{\stable}/GL(A).
\]
In particular, $\LoopTwo (\Gr (n,n+N))$ is a smooth affine algebraic
variety of dimension 
\[
(n+N)d.
\]
\end{theorem}

For the Lagrangian and isotropic orthogonal Grassmannians we have the following.

\begin{theorem}\label{thm: quiver description of Loop2Alg(LGr)}
Let $A$ be a vector space of dimension $d$ equipped with a split
non-degenerate orthogonal form $\phi :A\to A^{\vee}$ (so
$\phi^{\vee}=\phi$) and let $U$ be a vector space of dimension
$n$. Let
\[
V=V_{n,d}=\left\{(\alpha ,j)\in \Hom (A,A)\oplus \Hom (A,U):
\alpha^{\vee}\phi =\phi \alpha \right\} 
\]
and let 
\[
O(A,\phi ) = \left\{g\in GL(A):g^{\vee}\phi g=\phi  \right\}
\]
be the group of orthogonal automorphisms of $A$. Then $O(A,\phi )$ acts
naturally on $V$.  Let $V^{\stable}\subset V$ be the stable
locus (in the sense of GIT). Then
\[
\LoopTwo (\LGr(n,2n))\cong V^{\stable}/O(A,\phi ). 
\]
In particular, $\LoopTwo (\LGr (n,2n))$ is a smooth affine algebraic
variety of dimension 
\[
(n+1)d.
\]
\end{theorem}

The statement for the isotropic orthogonal Grassmannian is similar:

\begin{theorem}\label{thm: quiver description of Loop2Alg(OGr)}
Let $A$ be a vector space of dimension $d$ equipped with a
non-degenerate symplectic form $\phi :A\to A^{\vee}$ (so
$\phi^{\vee}=-\phi$) and let $U$ be a vector space of dimension
$n$. Let
\[
V=V_{n,d}=\left\{(\alpha ,j)\in \Hom (A,A)\oplus \Hom (A,U):
\alpha^{\vee}\phi =\phi \alpha \right\} 
\]
and let 
\[
Sp(A,\phi ) = \left\{g\in GL(A):g^{\vee}\phi g=\phi  \right\}
\]
be the group of symplectic automorphisms of $A$. Then $Sp(A,\phi )$ acts
naturally on $V$.  Let $V^{\stable}\subset V$ be the stable
locus (in the sense of GIT). Then
\[
\LoopTwo (\OGr(n,2n))\cong V^{\stable}/Sp(A,\phi ). 
\]
In particular, $\LoopTwo (\OGr (n,2n))$ is a smooth affine algebraic
variety of dimension 
\[
(n-1)d.
\]
\end{theorem}

\begin{remark}
In the case of the isotropic orthogonal Grassmannian,
Theorem~\ref{thm: quiver description of Loop2Alg(OGr)} 
implies that $d$ is even, since $A$ is equipped with a
non-degenerate symplectic form. In particular, the theorem implies
that $\LoopTwo(\OGr(n,2n))$ is empty for odd $d$.
\end{remark}

The quivers associated to our theorems are given in the following figure (Theorem~\ref{thm: quiver description of Loop2alg(Gr(n,n+N))} for the quiver on the left, and Theorems~\ref{thm: quiver description of Loop2Alg(LGr)} and~\ref{thm: quiver description of Loop2Alg(OGr)} for the quiver on the right):

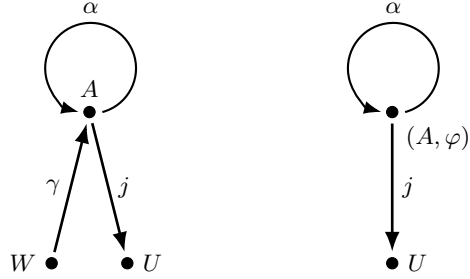
\begin{center}
\begin{tikzpicture}

\tikzset{myarrow/.style={-Latex, thick,line width=1pt,shorten >=2pt, shorten <=2pt}}

    \def\yDist{1.0} 
    \def\loopRadius{0.6} 

    \node[draw, circle, fill, inner sep=1.5pt, label=above:{\small
    $A$}] (L1) at (-2,\yDist) {}; 
    \node[draw, circle, fill, inner sep=1.5pt, label=right:{\small
    $U$}] (L2) at (-1.5,-\yDist) {}; 
    \node[draw, circle, fill, inner sep=1.5pt, label=left:{\small
    $W$}] (L3) at (-2.5,-\yDist) {}; 
    \draw[myarrow] (L1) -- node[right] {\small $j$} (L2);
    \draw[myarrow] (L3) -- node[left] {\small $\gamma $} (L1);

     \draw[-Latex,thick] (-1.83,\yDist) arc (-75:255:\loopRadius) node[midway,
     above] {\small $\alpha$}; 

    \node[draw, circle, fill, inner sep=1.5pt, label=below right:{\small
    $(A,\varphi)$}] (R1) at (2,\yDist) {}; 
    \node[draw, circle, fill, inner sep=1.5pt, label={right:{\small
    $U$}}] (R2) at (2,-\yDist) {}; 
    \draw[myarrow] (R1) -- node[right] {\small $j$} (R2);

     \draw[-Latex,thick] (2.17,\yDist) arc (-75:255:\loopRadius) node[midway,
     above] {\small $\alpha$}; 

\end{tikzpicture}
\end{center}

\subsection{Topological consequences}\label{subsec: topological consequences}

We now explain how our quiver descriptions of the algebraic loops
spaces of  $\Gr (n,n+N)$, $\OGr(n,2n)$, and $\LGr(n,2n)$
lead to proofs of the Bott periodicity equivalences given by
Equations~\eqref{eqn: Omega2BU=BU}, \eqref{eqn: Omega2O/U=BSp}, and
\eqref{eqn: Omega2Sp/U=BO} respectively. Full details of the argument
will be supplied in \S \ref{sec: topological details}. 

In \cite{Larson-Vakil-BottPeriodicity}, Larson and Vakil define a
homotopy category of algebraic spaces $\HoG $, whose objects can be
given sequences
\[
\dotsb \to  X_{i} \to  X_{i+1}\to \dotsb 
\]
of increasingly highly connected maps of algebraic spaces or stacks. Objects in
their category have a topological realization which is a well-defined
homotopy type that one can think of as $\lim_{i\to \infty} X_{i}$ in
the complex analytic topology with the direct limit
topology. Isomorphisms in the category include compatible systems of
maps $X_{i} \to Y_{i}$ which are affine bundles, or which are more and
more highly connected as $i\to \infty$. The category $\HoG$ provides
an algebro-geometric arena where isomorphisms induce homotopy
equivalences on the complex analytic counterparts. It is particularly
well adapted for the algebraic loop spaces and the quiver descriptions
given in this paper.

The natural sequence
\[
\dotsb \hookrightarrow \Gr (n,n+N)\hookrightarrow \Gr
(n+1,n+N+2)\hookrightarrow \dotsb
\]
is an example of an object in $\HoG$. The topological realization of
this object has homotopy type
\[
\lim_{n,N\to \infty} U(n+N)/(U(n)\times U(N)) = BU.
\]
Applying $\LoopTwo (-)$ to the above sequence we get a sequence
$\HoG$ 
\begin{equation}\label{eq:vancouver}
\dotsb \hookrightarrow \LoopTwo (\Gr (n,n+N))\hookrightarrow \LoopTwo (\Gr (n+1,n+N+2))\hookrightarrow \dotsb 
\end{equation}
which we call $\LoopTwo (BU)$.  In order to interpret \eqref{eq:vancouver} as an object of $\HoG$, we need to verify that the maps are increasingly connected in the sense of \cite{Larson-Vakil-BottPeriodicity}, which will fall out of what we are about to do.

Consider
$$\xymatrix{
\dotsb  \ar@{^(->}[r] &  V^{\stable}_{n,N,d}/GL_{d} \ar@{^(->}[d]  \ar@{^(->}[r] &
V^{\stable}_{n+1,N+1,d}/GL_{d} \ar@{^(->}[r] \ar@{^(->}[d]  & \dotsb \\
\dotsb   \ar@{^(->}[r]&  V_{n,N,d}/GL_{d} \ar[d]  \ar@{^(->}[r] &
V_{n+1,N+1,d}/GL_{d} \ar@{^(->}[r] \ar[d]  & \dotsb \\
\dotsb  \ar[r]^=  &  pt/GL_{d}  \ar[r]^= &
pt/GL_{d}  \ar[r]^=   & \dotsb \\
}
$$
The top row is just \eqref{eq:vancouver}, by Theorem~\ref{thm: quiver description of Loop2alg(Gr(n,n+N))}.  The inclusions from the first row into the second row are increasingly connected (they involve removing loci of increasingly higher codimension).  The morphisms for the second row to the third row are affine bundles, and hence infinitely connected in the sense of \cite{Larson-Vakil-BottPeriodicity}.  The horizontal morphisms in the third row are all equalities and the third row is $BGL(d)$ (as an element of $\HoG$).  The seocnd and first rows are tus isomorphic to $BGL(d)$ (in $\HoG$).  In particular, \eqref{eq:vancouver} is an element of $\HoG$, and is isomorphic to $BGL(d)$.  We explain this in detail in \S \ref{sec: topological details}.

Thus
Theorem~\ref{thm: quiver description of Loop2alg(Gr(n,n+N))} gives
rise to an isomorphism in $\HoG$
\[
\LoopTwo (BU)\cong BGL_{d}. 
\]
Taking the topological realizations of the above isomorphism yields a
homotopy equivalence between the direct limit of the varieties $\LoopTwo
(\Gr (n,n+N))$ (with the complex analytic topology) and $BU(d)$:
\[
\lim_{n,N\to \infty} \LoopTwo (\Gr (n,n+N)) \homotopyeq BU(d).
\]

It is not hard to define a sequence of maps that increase $d$ as well as $n$ and
$N$:
\[
\dotsb \hookrightarrow \LoopTwo (\Gr (n,n+N))\hookrightarrow
\Omega^{2}_{d+1,\alg} (\Gr (n+1,n+N+2))\hookrightarrow \dotsb
\]
which we denote $\Omega^{2}_{\infty ,\alg}(BU)$. It is a consequence
of the Boyer-Mann-Hurtubise-Milgram theorem (Theorem~\ref{thm:
Mann-Milgram algebraic loop approximation holds for generalized flag
varieties}) that the topological realization of $\Omega^{2}_{\infty
,\alg}(BU)$ has the homotopy type of $\LoopTwoTop (BU)$ (for any fixed
$d$ since they are all homotopy equivalent). Thus taking the
topological realization of the isomorphism
in $\HoG$ 
\[
\Omega^{2}_{\infty ,\alg } (BU)\cong BGL. 
\]
proves the homotopy equivalence
\[
\LoopTwoTop (BU)\homotopyeq BU .
\]

A nearly identical argument may be applied to the isomorphisms in
Theorems~\ref{thm: quiver description of Loop2Alg(LGr)} and \ref{thm:
quiver description of Loop2Alg(OGr)}. Briefly, applying $\LoopTwo (-)$
to the sequence
\[
\dotsb \hookrightarrow Sp(2n)/U(n) \hookrightarrow Sp(2n+2)/U(n+1) \hookrightarrow \dotsb 
\]
we get an object $\LoopTwo (Sp/U)$ in $\HoG$ which by applying
Theorem~\ref{thm: quiver description of Loop2Alg(LGr)} and applying a
argument similar to the one above yields an isomorphism in $\HoG$:
\[
\LoopTwo (Sp/U) \cong BO(d) . 
\]
Similarly, Theorem~\ref{thm: quiver description of Loop2Alg(OGr)}
yields an isomorphism in $\HoG $:
\[
\LoopTwo (O/U) \cong BSp(d) . 
\]

Taking topological realizations of the above two isomorphisms yields
homotopy equivalences for the direct limits of the algebraic double
loop varieties:
\[
\lim_{n\to \infty} \LoopTwo (Sp(2n)/U(n)) \homotopyeq BO(d),\quad
\quad \lim_{n\to \infty} \LoopTwo (O(2n)/U(n)) \homotopyeq BSp(d)
\]

Finally, by constructing stabilization maps that increase $d$ as well
as $n$, we get isomorphisms
\[
\Omega^{2}_{\infty ,\alg } (Sp/U) \cong BO, \quad \Omega^{2}_{\infty ,\alg } (O/U) \cong BSp 
\]
whose topological realizations give (thanks to the
Boyer-Mann-Hurtubise theorem) the homotopy equivalences
\[
\LoopTwoTop (Sp/U)\homotopyeq BO, \quad \LoopTwoTop (O/U)\homotopyeq
BSp
\]
which are equations~\eqref{eqn: Omega2Sp/U=BO} and \eqref{eqn:
Omega2O/U=BSp} of the real Bott periodicity equivalences.

See Section~\ref{sec: topological details} for precise statements of
the above results and their proofs.

\section{Zariski framed sheaves, quivers, and the proof of
Theorem~\ref{thm: quiver description of
Loop2alg(Gr(n,n+N))}}\label{sec: Zariski framed sheaves, quivers, and
proof of Loop2(Gr) theorem}

The material in this section has previously appeared in various guises
(e.g. \cite[\S 5]{Larson-Vakil-BottPeriodicity},
\cite[\S 3]{Nakajima-Handsaw}), but we include it to fix our notation
and orient the reader to our point of view which will be integral to
subsequent sections.

We fix homogeneous coordinates $(x_{0}:x_{1})$ on $\PP^{1}$ and we let
$p_{\infty }=(1:0)$. Let $V$ be a vector space. We use the notation
\[
\UL{V}(i) = \OO_{\PP^{1}}(i)\otimes V
\]
and in general, we use underlines to denote sheaves and maps of
sheaves, and we drop the underline to denote the induced map on cohomology.
(By the key Str\o mme-Beilinson Lemma~\ref{lem: Stromme}, the reader will appreciate this convention.)

Let $U$ and $W$ be vector spaces of dimension $n$ and $N$, respectively
and let
\[
\Gr (n,n+N) = \Gr (n,U\oplus W)
\]
be the Grassmannian of $n$-dimensional quotients of $U\oplus W$ with
basepoint given by the projection $U\oplus W\to U$. By pulling back
the universal quotient, we see that $\LoopTwo (\Gr (n,n+N))$
parameterizes exact sequences:

\[
\begin{tikzcd}
0\arrow[r]& F \arrow[r]& \UL{U}\oplus \UL{W} \arrow{r}{\UL{f} } &E
\arrow[r]& 0.
\end{tikzcd}
\]
Here $E$ is a rank $n$, degree $d$ bundle on $\PP^{1}$ and
$E|_{p_{\infty}}\cong U$ via the first component of $\UL{f}$ which we
call $\UL{\theta}$. The second component of $\UL{f}$ is zero at
$p_{\infty}$ and hence factors as
\[
\begin{tikzcd}
 \UL{W} \arrow[r,"\UL{\gamma }"]& E(-p_{\infty})
 \arrow[r,"\UL{x}_{1}"] &E. 
\end{tikzcd}
\]
and so we may write
\[
\UL{f}
=({\UL{\theta}},\UL{x}_{1}\UL{\gamma })
\]

By definition, the map $\UL{\theta}:\UL{U}\to E$ induces the framing
at $p_{\infty}$: $E|_{p_{\infty}}\cong U$ and moreover, it induces a
trivialization of $E$ on some (unspecified) Zariski open neighborhood
of $p_{\infty}$. We call such a structure a \emph{Zariski framing} of $E$.  Precisely, we come to
a central definition:  

\begin{definition}\label{defn: Zariski framing}
Let $E$ be a rank $n$ sheaf on $\PP^{1}$ and let $U$ be an
$n$-dimensional vector space. A sheaf map 
\[
\UL{\theta}: \UL{U}\to E
\]
is called a \emph{Zariski framing} if $\UL{\theta}|_{p_{\infty}}:U \to
E|_{p_{\infty}}$ is an isomorphism.
\end{definition}

The following theorem identifies the stack of Zariski framed sheaves
with the stack of representations of the following quiver:

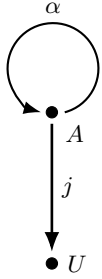
\begin{center}
\begin{tikzpicture}

\tikzset{myarrow/.style={-Latex, thick,line width=1pt,shorten >=2pt, shorten <=2pt}}

\def\yDist{1.0} 
\def\loopRadius{0.6} 

\node[draw, circle, fill, inner sep=1.5pt, label=below right:{\small $A$}] (R1) at (0,\yDist) {}; 
\node[draw, circle, fill, inner sep=1.5pt, label={right:{\small $U$}}] (R2) at (0,-\yDist) {}; 
\draw[myarrow] (R1) -- node[right] {\small $j$} (R2);

\draw[-Latex,thick] (0.17,\yDist) arc (-75:255:\loopRadius) node[midway, above] {\small $\alpha$}; 

\end{tikzpicture}
\end{center}

\begin{theorem}\label{thm: stack of Zariski framed sheaves is
equivalent to stack of quiver representation}
The stack of degree $d$ sheaves $E$ equipped with a Zariski framing
$\UL{\theta}: \UL{U}\to E$ is equivalent to the stack quotient
\[
[\,\,\Hom (A,A)\oplus \Hom (A,U) \,/\, GL(A)\,\, ]
\]
where $A$ is a vector space of dimension $d$. Moreover, the open
substack where $E$ is locally free is a scheme and is given by the
open condition where 
\[
(\alpha ,j)\in \Hom (A,A)\oplus \Hom (A,U) 
\]
satisfy 
\begin{equation}\label{eqn: open condition involving ker j}
\Ker (\alpha -\lambda \cdot \Id_{A})\cap \Ker (j) = 0\text{ for all }\lambda \in \overline{\kk }
\end{equation}
\end{theorem}

\proof Since $\UL{U}$ is torsion free and $\UL{\theta}$ is generically
an isomorphism,
\[
\Ker \UL{\theta}=0\text{ and } \Coker \UL{\theta}=Q ,
\]
a zero-dimensional length $d$ sheaf supported on $\AA^1=\PP^{1}-p_{\infty}$.

The sequence
\[
\begin{tikzcd}
0\arrow[r]& \UL{U} \arrow{r}{\UL{\theta}} & E\arrow[r]&Q\arrow[r]&0
\end{tikzcd}
\]
then induces the sequence
\[
\begin{tikzcd}
0\arrow[r]& {U} \arrow{r}{{\theta}} &H^{0}(E)\arrow[r]&H^{0}(Q)\arrow[r]&0.
\end{tikzcd}
\]
Using $\UL{\theta}|_{p_{\infty}}$ to identify $E|_{p_{\infty}}$ with
$U\otimes \OO_{p_{\infty}}$ we may write
\[
\begin{tikzcd}
0\arrow[r]&E(-p_{\infty})\arrow{r}{\cdot \UL{x}_{1}}
&E\arrow[r]&U\otimes \OO_{p_{\infty}} \arrow[r] &0
\end{tikzcd}
\]
and hence
\[
\begin{tikzcd}
0\arrow[r]&H^{0}(E(-p_{\infty}))\arrow{r}{\cdot {x}_{1}}
&H^{0}(E)\arrow[r]&U \arrow[r] &0.
\end{tikzcd}
\]
The above sequence is split by $\theta$ so writing
$A=H^{0}(E(-p_{\infty}))$, we have identified
\begin{align*}
H^{0}(E) &= A\oplus U,\\
H^{0}(E(-1))&= A,\\
H^{0}(Q)&= A.
\end{align*}
Then the sheaf map
\[
\begin{tikzcd}
E(-1)\arrow{r}{\cdot \UL{x}_{0}}
&E
\end{tikzcd}
\]
induces a map in cohomology
\[
\begin{tikzcd}
A\arrow{r}{\cdot x_{0}} &A\oplus U
\end{tikzcd}
\]
whose components we denote by $\alpha$ and $j$:
\[
\alpha :A\to A, \quad j: A\to U.
\]
The above construction induces a map 
\[
[\UL{\theta}:\UL{U}\to E]\mapsto (\alpha ,j)
\]
from the moduli stack of Zariski framed sheaves to the moduli stack of
quiver representations. To show this is an equivalence, we need a
reverse construction. This is proved via the following Lemma which can
be attributed to Str\o mme \cite[Prop~1.1]{Stromme} (although it is a
special case of the Beilinson spectral sequence).
\begin{lemma}\label{lem: Stromme}
Let $E$ be a sheaf on $\PP^{1}$ with $H^{1}(E(-1))=0$. Then the following
sequence is exact:
\[
\begin{tikzcd}
0\arrow{r} & \UL{H^{0}(E(-1))}(-1)
\arrow{rr}{\UL{x}_{1}x_{0}-\UL{x}_{0}x_{1}}& &
\UL{H^{0}(E)} \arrow{r} & E\arrow[r] & 0
\end{tikzcd}
\]
where recalling our notation, sheaf maps are denoted with
underlines and their induced maps on cohomology are denoted without,
so in particular $\UL{x}_{i}:\OO (-1)\to \OO$ and
$x_{i}:H^{0}(E(-1))\to H^{0}(E)$ in the above.
\end{lemma}
\begin{proof}
Let $\PP^{1}\times \PP^{1}$ have coordinates
$((y_{0}:y_{1}),(x_{0}:x_{1}))$ with $p,q:\PP^{1}\times \PP^{1}\to
\PP^{1}$ projections onto the first and second factor respectively. Let 
\[
\Delta =\{y_{1}x_{0}-y_{0}x_{1} =0 \}\subset \PP^{1}\times \PP^{1}
\]
be the diagonal. Then we get an isomorphism
\[
E\cong q_{*}\left(p^{*}E\otimes \OO_{\Delta} \right)
\]
where we have implicitly identified the two factors of $\PP^{1}$ by
$x_{i}=y_{i}$.  We then tensor the exact sequence
\[
\begin{tikzcd}
0\arrow[r] & \OO_{\PP^{1}\times \PP^{1}}(-1,-1)
\arrow{rr}{y_{1}x_{0}-y_{0}x_{1}}& & \OO_{\PP^{1}\times \PP^{1}}
\arrow[r]& \OO_{\Delta} \arrow[r]& 0
\end{tikzcd}
\]
by $p^{*}E$ and apply the functor $q_{*}(-)$. By hypothesis, we have
$R^{1}q_{*}(p^{*}E(-1,-1))=0$, and so we get
\[
\begin{tikzcd}
0\arrow[r] &q_{*}(p^{*}E(-1,-1)) \arrow[r] & q_{*}(p^{*}E) \arrow[r] &
q_{*}(p^{*}E)\otimes \OO_{\Delta} \arrow[r] & 0 
\end{tikzcd}
\]
which can be written as 
\[
\begin{tikzcd}
0\arrow[r] &\UL{H^{0}(E(-1))}(-1)
\arrow{rr}{\UL{y_{1}}x_{0}-\UL{y_{0}}x_{1}} && \UL{H^{0}(E)} \arrow[r]
&E \arrow[r] &0
\end{tikzcd}
\]
which proves the lemma.
\end{proof}

We now use the identifications coming from the Zariski framing to
rewrite the two maps
\[
x_{0},x_{1}: H^{0}(E(-1)) \to H^{0}(E)
\]
as
\[
\begin{tikzcd}
A \arrow{r}{\begin{pmatrix} \alpha \\ j  \end{pmatrix}} &A\oplus U,
\quad & A \arrow{r}{\begin{pmatrix} \Id_{A} \\ 0  \end{pmatrix}} &A\oplus U
\end{tikzcd}
\]
respectively. Then the exact sequence from Lemma~\ref{lem: Stromme}
reads
\begin{equation}\label{eqn: A(-1)->A+U->E}
\begin{tikzcd}[column sep=3em]
  0
    \arrow[r]
  & \UL{A}(-1)
    \arrow[rr,"{%
      \begin{pmatrix}
        \UL{x}_{1}\alpha - \UL{x}_{0} \\[3pt]
        \UL{x}_{1}j
      \end{pmatrix}
    }"]
  && \UL{A}\oplus \UL{U}
    \arrow[r,"{(\UL{\beta},\UL{\theta})}"]
  & E
    \arrow[r]
  & 0
\end{tikzcd}
\end{equation}
and we see that $E$ can be reconstructed from $(\alpha ,j)$ as
the cokernel of $ \left(\begin{smallmatrix} \UL{x}_{1} \alpha -\UL{x}_{0}\\ 
\UL{x}_{1}j \end{smallmatrix} \right)$ and $\UL{\theta}:\UL{U}\to E$
can be reconstructed by restricting the quotient map to $\UL{U}$.

This proves the first part of the theorem. To prove the second part,
we observe that  $E$ is locally free if and only if
$ \left(\begin{smallmatrix} \UL{x}_{1} \alpha -\UL{x}_{0}\\ 
\UL{x}_{1}j \end{smallmatrix} \right)$ is bundle injective, namely for
each point $\lambda =\frac{x_{0}}{x_{1}}\in \PP^{1}-p_{\infty}$,
the vector space map ${\tiny \begin{pmatrix} \alpha -\lambda \Id_{A}\\
j \end{pmatrix}} : A\to A\oplus U $ is injective. This is equivalent
to condition \eqref{eqn: open condition involving ker j}.
\qed 

\begin{remark}\label{rem: linear control systems}
Although it has no bearing on this paper, it is amusing to note that
the stack in Theorem~\ref{thm: stack of Zariski framed sheaves is
equivalent to stack of quiver representation} is dual to the stack of
\emph{linear control systems} and the open condition \eqref{eqn: open
condition involving ker j} corresponds to so-called
\emph{controllable} linear control systems
\cite{Geiss-intro-to-moduli-quivers}.
\end{remark}

As discussed in the beginning of this section, points in $\LoopTwo
(\Gr (n,U\oplus W))$ are determined by the bundle $E$ and the
surjective map
\[
\begin{tikzcd}[column sep={8em,between origins}]
\UL{U}\oplus \UL{W} \arrow{r}{(\UL{\theta },\UL{x}_{1}\UL{\gamma})} &E .
\end{tikzcd}
\]
This is equivalent to the data of a Zariski framed bundle
$\UL{\theta}:\UL{U}\to E$ and a sheaf map $\UL{\gamma}:\UL{W}\to
E(-1)$ such that $\UL{x}_{1}\UL{\gamma}$ surjects onto $Q$, the
cokernel of $\UL{\theta}$. Since $E$ is the quotient of a trivial
bundle, it has no negative summands and hence
$H^{1}(E(-1))=0$. Consequently $\UL{\gamma}$ determines and is
determined by its map on global sections:
\[
\gamma :W \to A.
\]

From the exact sequence \eqref{eqn: A(-1)->A+U->E}, we see that the
support of $Q$ occurs at points where $\UL{x}_{1}\alpha
-\UL{x}_{0}=0$, in other words where $x_{0}/x_{1}=\lambda $ is an
eigenvalue of $\alpha$. Thus the condition that $\UL{W}\to Q$ is
surjective is equivalent to
\begin{equation}\label{eqn: open condition with gamma}
\im (\alpha -\lambda \cdot \Id_{A})+ \im (\gamma )=A \text{ for
all }\lambda\in \overline{\kk}.
\end{equation}

This discussion yields a version of Theorem~\ref{thm: quiver description
of Loop2alg(Gr(n,n+N))} which has a more explicit description of the open
condition:

\begin{proposition}\label{prop: quiver description of Loop2(Gr(n,n+N))
with explicit stability condition}
Let $U$, $W$, and $A$ be vector spaces of dimension $n$, $N$, and
$d$. Let $V_{n,N,d} = \Hom (A,A)\oplus \Hom (A,U)\oplus \Hom (W,A)$ and let
$V^{\open}_{n,N,d}\subset V_{n,N,d}$ be the open set of $(\alpha ,j,\gamma )\in V$
satisfying Equations~\eqref{eqn: open condition involving ker j} and
\eqref{eqn: open condition with gamma}, namely:
\begin{enumerate}
\item $\Ker (\alpha -\lambda \cdot \Id_{A})\cap \Ker (j) = 0$ for all
$\lambda\in \overline{\kk} $.
\item  $\im (\alpha -\lambda \cdot \Id_{A})+  \im (\gamma )=A$ for
all $\lambda\in \overline{\kk} $.
\end{enumerate}
Then the constructions of this section induce an isomorphism of affine
varieties:
\[
\LoopTwo (\Gr (n,n+N))\overset \sim \longleftrightarrow V^{\open}_{n,N,d}/GL(A) .
\]
\end{proposition}

\begin{remark}\label{rem: duality induced by Gr(n,n+N)=Gr(N,n+N)}
The isomorphism $\Gr (n,U\oplus W)\overset \sim \longleftrightarrow \Gr (N,U^{\vee}\oplus
W^{\vee})$ given by dualizing exact sequences induces an isomorphism
on the quiver side given by $(A,U,W)\mapsto
(A^{\vee},W^{\vee},U^{\vee})$ and $(\alpha ,j,\gamma )\mapsto
(\alpha^{\vee},\gamma^{\vee},j^{\vee})$. Note that the two open
conditions (1) and (2) are dual to each other under this equivalence. 
\end{remark}

To finish the proof of Theorem~\ref{thm: quiver description of
Loop2alg(Gr(n,n+N))}, it remains to show that $V^{\open}_{n,N,d} =
V^{\stable}_{n,N,d}$. Stability for quivers is well understood going
back to work of King \cite{King-quivers} and continuing with the work
of Nakajima \cite{Nakajima-Duke88,Nakajima-Handsaw}. For the general
theory, stability depends on a choice of a linearization which is
determined in our case by a character $\chi : GL(A)\to \mathbb{G}_{m}$ which
we may write $\chi (g) = \operatorname{det}(g)^{k}$. The three
different cases are $k<0$, $k=0$, $k>0$ and we are interested in the
trivial linearization, namely $k=0$.

Using the Jordan decomposition of $\alpha$, it is easy to see that
condition (2) is equivalent to the condition that there does not exist
a non-trivial proper subspace $A'\subset A$ which is invariant under
$\alpha$ and contains the image of $\gamma$. Similarly, condition (1)
is equivalent to the condition that there does not exist a non-trivial
proper subspace $A'\subset A$ which is invariant under $\alpha$ and
contained in the kernel of $j$.

Condition (2) is equivalent to GIT stability for the $k=1$
linearization \cite[Def.~2.2]{Nakajima-Handsaw}. As we observed in
Remark~\ref{rem: duality induced by Gr(n,n+N)=Gr(N,n+N)}, condition
(1) is dual to condition (2) and thus condition (1) is equivalent to
stability for the $k=-1$ linearization. For both $k=1$ and $k=-1$,
there are no strictly semi-stable points.

We are interested in the trivial linearization, $k=0$. In this case,
there are strictly semi-stable points, but the stable points are simple
representations: these are exactly those points satisfying both
conditions (1) and (2).

\section{Proof of Theorems~\ref{thm: quiver description of
Loop2Alg(LGr)} and \ref{thm: quiver description of
Loop2Alg(OGr)}}\label{sec: proofs of the Omega2(LG/OG) theorems}

Let 
\[
 \OGr (n,2n) = O(2n)/U(n), \quad \LGr (n,2n) = Sp(2n)/U(n)
\]
be the maximal isotropic orthogonal Grassmannian and the Lagrangian
Grassmannian respectively. These parameterize $n$-dimensional
isotropic quotients of a $2n$-dimensional vector space equipped with
an orthogonal and a symplectic form respectively. Although we have
written these spaces as quotients of a compact Lie group by a
subgroup, they are also given by quotients of an algebraic group by a
parabolic subgroup and are therefore projective algebraic varieties
defined over $\kk$. We will handle these two cases simultaneously and
so it will be convenient to introduce the following notations:
\[
X_{+,n} = OGr(n,2n), \quad X_{-,n} = LGr(n,2n).
\]

If $V$ is a $2n$-dimensional vector space equipped with a
non-degenerate orthogonal, respectively symplectic form $\omega$, then a
quotient
\[
\begin{tikzcd}
V \arrow[r,  two heads, "f"] & L
\end{tikzcd}
\]
is by definition \emph{maximal isotropic}, respectively \emph{Lagrangian} if $L$ is
$n$-dimensional and $\omega |_{\Ker f} = 0$. 

Associated to the form $\omega$ is an isomorphism
\[
J : V\to V^{\vee } 
\]
satisfying $J^{\vee} = J$, respectively $J^{\vee}=-J$. The following
lemma is fairly well known (e.g. \cite[\S 2]{Hitching-2007}):
\begin{lemma}\label{lem: lagrangian quotient condition as a ses}
A linear map $f:V\to L$ is a maximal isotropic or Lagrangian quotient if
and only if the following sequence is short exact:
\[
\begin{tikzcd}[column sep={6em,between origins}]
0 \arrow[r] &
L^{\vee } \arrow[r, "J^{-1}\circ f^{\vee}"] &
V \arrow[r, "f"] &
L \arrow[r] &
0
\end{tikzcd}
\]
\end{lemma}

Let $U$ be an $n$-dimensional vector space and let 
\[
J_{\pm}: U\oplus U^{\vee} \to U^{\vee}\oplus U
\]
where 
\[
J_{\pm} = \begin{pmatrix}
        0 & \pm 1_{U^{\vee }}\\
        1_{U} & 0
      \end{pmatrix}.
\]
Then $J_{\pm}$ induces an orthogonal/symplectic form on $U\oplus
U^{\vee}$ and so by the lemma, $\Xpm$ parameterizes short exact
sequences of the form
\[
\begin{tikzcd}[column sep={6em,between origins}]
0 \arrow[r] &
L^{\vee } \arrow[r, "J_{\pm }^{-1}\circ f^{\vee}"] &
U\oplus U^{\vee } \arrow[r, "f"] &
L \arrow[r] &
0.
\end{tikzcd}
\]

The space $\Xpm$ has a natural base point given by the sequence
\[
0\to U^{\vee}\to U\oplus U^{\vee} \to U\to 0
\]
where the maps are the natural inclusions and projections.

With the above discussion, we see that the algebraic double loop space
$\LoopTwo(\Xpm )$ parameterizes short exact sequences of
bundles on $\PP^{1}$ of the form
\begin{equation}\label{eqn: Evee-->U+Uvee-->E}
\begin{tikzcd}[column sep={6em,between origins}]
0 \arrow[r] &
E^{\vee } \arrow[r, "J_{\pm }^{-1}\circ \UL{f}^{\vee}"] &
\UL{U}\oplus \UL{U}^{\vee } \arrow[r, "\UL{f}"] &
E \arrow[r] &
0
\end{tikzcd}
\end{equation}
where the bundle $E$ has rank $n$ and degree $d$, and the map $\UL{f}$
restricted to $p_{\infty} = (1:0)$ is 0 on the $U^{\vee}$ factor. In
keeping with the notation of Section~\ref{sec: Zariski framed sheaves,
quivers, and proof of Loop2(Gr) theorem} we may then write
\[
\UL{f} = (\UL{\theta},\UL{x}_{1}\UL{\gamma})
\]
where $\UL{\theta}:\UL{U}\to E$ and $\UL{\gamma}:\UL{U}^{\vee}\to
E(-p_{\infty}).$ We then have
\[
J_{\pm}^{-1}\circ \UL{f}^{\vee} = \begin{pmatrix} \UL{\gamma}^{\vee}\UL{x}_{1} \\
\pm \UL{\theta}^{\vee}  \end{pmatrix}
\]
and we may re-express the condition $\UL{f}\circ J_{\pm}^{-1}\circ
\UL{f}^{\vee}=0$ as

\begin{equation}\label{eqn: theta gammadual x1 = mp x1 gamma thetadual}
\UL{\theta}\, \UL{\gamma}^{\vee}\, \UL{x}_{1} =  \mp \UL{x}_{1}\, 
\UL{\gamma}\,  \UL{\theta}^{\vee }.
\end{equation}

The above discussion has then proved the following intermediate result:

\begin{proposition}\label{prop: Loop2(Xpm) is a closed set in
Loop(Gr(n,2n)) satisfying condition involving theta and gamma}
$\LoopTwo (\Xpm )$ is the closed subset of $\LoopTwo (\Gr(n,2n))\cong
V^{\open}_{n,n,d}/GL(A)$ given by Equation~\eqref{eqn: theta gammadual x1
= mp x1 gamma thetadual}. 
\end{proposition}

Our strategy to prove Theorems \ref{thm: quiver description of
Loop2Alg(LGr)} and \ref{thm: quiver description of Loop2Alg(OGr)} is
to translate the sheaf theoretic condition given by
Equation~\eqref{eqn: theta gammadual x1 = mp x1 gamma thetadual} into
the language of the quiver data $(A,\alpha ,j,\gamma )$. This
translation will involve constructing a non-degenerate bilinear form
$\phi :A\to A^{\vee}$ such that the closed condition is equivalent to
the following equations
\begin{align}\label{eqns: closed conditions}
\phi^{\vee}&= \mp \phi \nonumber \\ 
\phi \alpha &=\alpha^{\vee}\phi \\
\phi \gamma &= - j^{\vee}. \nonumber 
\end{align}
This will allow us to eliminate $\gamma$ from the quiver data and
reduce the structure group from $GL(A)$ to $GL(A,\phi )$.

We begin by defining $\phi :A\to A^{\vee}$. Recall that $A =
H^{0}(E(-1))$. Tensoring the sequence \eqref{eqn:
Evee-->U+Uvee-->E} by $\OO (-1)$ we get
\[
0\to E^{\vee}(-1)\to (\UL{U}\oplus \UL{U}^{\vee})(-1)\to E(-1)\to 0.
\]
Since $\OO (-1)$ and hence $(\UL{U}\oplus \UL{U}^{\vee})(-1)$ has $H^0=H^1=0$, the long exact
sequence associated to the above short exact sequence gives an
 isomorphism $H^{0}(E(-1))\overset \sim \longrightarrow H^{1}(E^{\vee}(-1))$. We define
$\phi$ to be the composition of this isomorphism with the Serre
duality isomorphism: 
\[
\phi : A = H^{0}(E(-1))\to H^{1}(E^{\vee}(-1)) \to H^{0}(E(-1))^{\vee}=A^{\vee}.
\]

The following seems almost a miracle.

\begin{lemma}\label{lem: phivee = mp phi}
The  isomorphism $\phi: A \rightarrow A^\vee$ satisfies $\phi^{\vee} = \mp \phi$. 
\end{lemma}
\begin{proof}
Grothendieck-Verdier duality \cite[Eq~3.20]{Huybrechts-FMbook} states
that if $\pi :X\to Y$ is a morphism of smooth schemes and the
dualizing functor on $D^{b}(\operatorname{Coh}(X))$ is defined by 
\[
\DD_{X}(F^{\bullet}) = (F^{\bullet})^{\vee} \otimes \omega_{X}[\dim X ]
\]
then
\[
\Rpistar \circ \DD_{X} \homotopyeq  \DD_{Y}\circ \Rpistar .
\]

We let $Y=pt$ and $X=\PP^{1}$ and let 
\[
F^{\bullet} = [\UL{U}(-1)\stackrel{\UL{\theta}}{\longrightarrow } E(-1)]
\]
where this two term complex is supported in degrees $-1$ and $0$. Then
\[
\DD_{\PP^{1}}F^{\bullet} = [E^{\vee}(-1)\stackrel{\UL{\theta}^{\vee
}}{\longrightarrow } \UL{U}^{\vee}(-1)] .
\]

Note that $\Rpistar F^{\bullet}=A$ a vector space considered as a
complex supported in degree 0.

We will repeatedly use the following general fact from homological
algebra:
\bigskip

\begin{quote}
\emph{A short exact sequence with a split middle term is equivalent to
a quasi-isomorphism of a pair of two term complexes.}  (A short
exact sequence with a split middle term is the total complex of a 2x2
double complex and since the total complex is acyclic, the spectral
sequence of the double complex implies the vertical maps induce a
quasi-isomorphism of the rows.)
\end{quote}
\bigskip

So for example, the short exact sequence

\[
\begin{tikzcd}[column sep=6em]
0 \arrow[r]
& E^{\vee}
  \arrow[r,"{%
    \begin{pmatrix}
      \UL{\gamma}^{\vee}\UL{x}_{1} \\[2pt]
      \pm\,\UL{\theta}^{\vee}
    \end{pmatrix}
  }"]
& \UL{U}\oplus \UL{U}^{\vee}
  \arrow[r,"{(\:\UL{\theta}\,,\,\UL{x}_{1}\UL{\gamma}\:)}"]
& E \arrow[r]
& 0
\end{tikzcd}
\]

is equivalent to the following quasi-isomorphism of two-term complexes:

\[
\begin{tikzcd}[column sep=6em,row sep=4em]
  \UL{U}(-1)
    \arrow[r,"\UL{\theta }"]
  & E(-1)
  & F^{\bullet}
    \arrow[d,bend left=25,"\Phi=\Psi^{-1}   "]\\
  E^{\vee}(-1)
    \arrow[u,"\mp \UL{\gamma}^{\vee}\UL{x}_{1}"]
    \arrow[r,"\UL{\theta }^{\vee}"]
  & \UL{U}^{\vee }(-1)
    \arrow[u,"\UL{x}_{1}\UL{\gamma }"]
  & \DD F^{\bullet}
    \arrow[u,"\Psi  "]
\end{tikzcd}
\]

The map of complexes on the left above is a quasi-isomorphism and
hence induces an isomorphism $\Psi :\DD F^{\bullet}\to F^{\bullet}$
in the derived category whose inverse we denote $\Phi$.

We remark that going from a short exact sequence to a map of two term
complexes introduces a sign which in the above is manifested by the
$\mp$ appearing on the left vertical arrow (as opposed to the $\pm$
appearing on the $\UL{\theta}^{\vee}$ in the sequence). This crucial
sign change is the origin of the form $\phi$ having the opposite sign
as $J_{\pm}$.

We apply $\DD (-)$ to the complex to get

\[
\begin{tikzcd}[column sep=6em,row sep=4em]
  \UL{U}^{\vee }(-1)
    \arrow[d,"\mp\UL{x}_{1} \UL{\gamma}"]
  & E^{\vee}(-1)
    \arrow[d,"\UL{\gamma }^{\vee } \UL{x}_{1}"]
    \arrow[l,"\UL{\theta }^{\vee}"]
  & \DD F^{\bullet}
    \arrow[d,"\DD \Psi  "]
  \\
  E(-1)
  & \UL{U}(-1)
    \arrow[l,"\UL{\theta }"]
  & F^{\bullet}
\end{tikzcd}
\]
Rotating the diagram on the left by 180 degrees we see that $\DD \Psi
=\mp \Psi$ and consequently we have $\DD \Phi =\mp \Phi$. Applying
$\Rpistar $ to $\Phi$ we get
\[
\Rpistar \Phi : \Rpistar F^{\bullet} = A \to \Rpistar
\DD_{\PP^{1}}F^{\bullet}=\DD_{pt}\Rpistar F^{\bullet} = A^{\vee}.
\]
Unpacking the definitions we see that $\phi  = \Rpistar \Phi$ and so
\begin{align*}
\phi &=\Rpistar \Phi \\
&=\mp \Rpistar \DD_{\PP^{1}}\Phi \\
&=\mp \DD_{pt}\Rpistar \Phi \\
&=\mp \DD_{pt}\phi \\
&=\mp \phi^{\vee}
\end{align*}
which completes the proof of the Lemma.
\end{proof}

\begin{lemma}\label{lem: phi.alpha=alphavee.phi} The following
equation holds:
\[
\phi \alpha =\alpha^{\vee}\phi .
\]
\end{lemma}

\proof We let

\[
\UL{\sigma} = \UL{x}_{1}\alpha -\UL{x}_{0}
\]
and we twist the exact sequence \eqref{eqn: A(-1)->A+U->E} by $\OO
(-1)$ to obtain:

\[
\begin{tikzcd}[column sep=3em]
  0
    \arrow[r]
  & \UL{A}(-2)
    \arrow[r,"{%
      \begin{pmatrix}
        \UL{\sigma } \\
        \UL{x}_{1}j
      \end{pmatrix}
    }"]
  & (\UL{A}\oplus \UL{U})(-1)
    \arrow[r,"{(\UL{\beta},\UL{\theta})}"]
  & E(-1)
    \arrow[r]
  & 0.
\end{tikzcd}
\]

We now use the same homological algebra trick as in the proof of
Lemma~\ref{lem: phivee = mp phi} applied to the above to obtain:

\[
\begin{tikzcd}[column sep=6em,row sep=4em]
  \UL{A}(-2)
    \arrow[d,"-\UL{x}_{1}j"]
    \arrow[r,"\UL{\sigma  }"]
  & \UL{A}(-1)
    \arrow[d,"\UL{\beta }"]
  & A^{\bullet }
    \arrow[d,"\JJ "]
  \\
  \UL{U}(-1)
    \arrow[r,"\UL{\theta }"]
  & E(-1)
  & F^{\bullet}
\end{tikzcd}
\]

We compose the derived category isomorphism $\JJ$ with the
isomorphisms $\Phi$ and $\DD \JJ$ to get an isomorphism
\[
\DD \JJ \circ \Phi \circ \JJ :A^{\bullet}\longrightarrow \DD A^{\bullet }
\]
which is realized by the following diagram of quasi-isomorphisms:

\begin{equation}\label{eqn: big diagram}
\begin{tikzcd}[column sep=6em,row sep=4em]
  \UL{A}(-2)
    \arrow[d,"-\UL{x}_{1}j"]
    \arrow[r,"\UL{\sigma  }"]
  & \UL{A}(-1)
    \arrow[d,"\UL{\beta }"]
  & A^{\bullet }
    \arrow[d,"\JJ "]
  \\
  \UL{U}(-1)
    \arrow[r,"\UL{\theta }"]
  & E(-1)
  & F^{\bullet}
    \arrow[d,"\Phi"]\\
  E^{\vee}(-1)
    \arrow[d,"\UL{\beta}^{\vee}"]
    \arrow[u,"\mp \UL{\gamma}^{\vee}\UL{x}_{1}"]
    \arrow[r,"\UL{\theta }^{\vee}"]
  & \UL{U}^{\vee }(-1)
    \arrow[d,"-\UL{x}_{1}j^{\vee}"]
    \arrow[u,"\UL{x}_{1}\UL{\gamma }"]
  & \DD F^{\bullet}
    \arrow[d,"\DD \JJ "]\\
    \UL{A}^{\vee}(-1) 
    \arrow[r,"\UL{\sigma  }^{\vee}"]
  & \UL{A}^{\vee }
  & \DD A^{\bullet} 
\end{tikzcd}
\end{equation}

Recall that 
\[
\coker \UL{\theta} = Q
\]
is a zero-dimensional sheaf supported on
$\PP^{1}-p_{\infty}=\operatorname{Spec}\kk [z]$ where
$z=x_{0}/x_{1}$. Considering the sheaf $Q$ as an object in
$D^{b}(\PP^{1})$ (concentrated in degree 0), we see that 
\[
F^{\bullet}\cong Q
\]
and consequently via $\JJ$ we also have 
\[
A^{\bullet}\cong Q.
\]
Let $\UL{\phi}=\DD \JJ \circ \Phi \circ \JJ$ and by the above diagrams
we see
\[
\UL{\phi}: Q\longrightarrow Q^{\vee}
\]
is an isomorphism. As the notation suggests, $\Rpistar \UL{\phi}=\phi$
(since $\Rpistar \JJ =1_{A}$).

Now since $Q$ and $Q^{\vee}$ are sheaves on $\operatorname{Spec}\kk
[z]$, we may regard the isomorphism $\UL{\phi }$ as a $\kk [z]$-module
isomorphism.

We note that the underlying vector space of the $\kk [z]$-module $Q$ is
$A$ and since $\UL{x}_{1}\alpha -\UL{x}_{0}=0$ on $\coker
\UL{\sigma}=Q$, we see that multiplication by $z=x_{0}/x_{1}$ is given
by $\alpha :A\to A$. Similarly, the underlying vector space of the
module $Q^{\vee}$ is $A^{\vee}$ and multiplication by $z$ is given by
$\alpha^{\vee}:A^{\vee}\to A^{\vee}$.

Since $\UL{\phi}:Q\to Q^{\vee}$ is a module isomorphism, the
underlying linear map of vector spaces $\phi :A\to A^{\vee}$ commutes
with multiplication by $z$, namely
\[
\phi \alpha =\alpha^{\vee}\phi 
\]
which proves the lemma.\qed 

\begin{lemma}\label{lem: phi.gamma = -jvee}
The following holds
\[
\phi \gamma =-j^{\vee}.
\]
\end{lemma}
\proof
First, we note that the isomorphism $\UL{\phi}:A^{\bullet}\to \DD
A^{\bullet}$ can be realized as the following map of complexes:

\[
\begin{tikzcd}[column sep=6em,row sep=4em]
  \UL{A}^{\vee }
    \arrow[r,"\UL{\sigma}^{\vee}"]
  & \UL{A}^{\vee}(1)
  & \DD A^{\bullet }
  \\
  \UL{A}(-1)
    \arrow[r,"\UL{\sigma  }"]
    \arrow[u,"\phi \UL{x}_{1}"]
  & \UL{A}
    \arrow[u,"\phi \UL{x}_{1}"]
  & A^{\bullet}
    \arrow[u,"\UL{\phi } "]
\end{tikzcd}
\]
since the induced sheaf isomorphism $\coker \UL{\sigma}\to \coker
\UL{\sigma}^{\vee}$ agrees with $\UL{\phi}$ on the underlying vector
spaces $\phi :A\to A^{\vee}$.

Now by adding the cokernels to the complexes in the diagram
\eqref{eqn: big diagram}, we get a commutative diagram where the rows
are short exact sequences. Then tensoring this big diagram by $\OO
(1)$ and adding at the bottom the above complex (and cokernel) we get the
following:

\begin{equation*}
\begin{tikzcd}[column sep=6em,row sep=4em]
  \UL{A}(-1)
    \arrow[d,"-\UL{x}_{1}j"]
    \arrow[r,"\UL{\sigma  }"]
  & \UL{A}
    \arrow[d,"\UL{\beta }"]
    \arrow[r]
  & \coker \UL{\sigma }
    \arrow[d]
  \\
  \UL{U}
    \arrow[r,"\UL{\theta }"]
  & E
    \arrow[r]
  & \coker \UL{\theta }
  \\
  E^{\vee}
    \arrow[d,"\UL{\beta}^{\vee}"]
    \arrow[u,"\mp \UL{\gamma}^{\vee}\UL{x}_{1}"]
    \arrow[r,"\UL{\theta }^{\vee}"]
  & \UL{U}^{\vee }
    \arrow[r]
    \arrow[d,"-\UL{x}_{1}j^{\vee}"]
    \arrow[u,"\UL{x}_{1}\UL{\gamma }"]
  & \coker \UL{\theta}^{\vee}
    \arrow[d]
    \arrow[u]
  \\
    \UL{A}^{\vee} 
    \arrow[r,"\UL{\sigma  }^{\vee}"]
  & \UL{A}^{\vee}(1)
    \arrow[r]
  & \coker \UL{\sigma}^{\vee}
  \\
  \UL{A}(-1)
    \arrow[u,"\phi \UL{x}_{1}"]
    \arrow[r,"\UL{\sigma  }"]
  & \UL{A}
    \arrow[u,"\phi \UL{x}_{1}"]
    \arrow[r]
  & \coker \UL{\sigma }
    \arrow[u]    
\end{tikzcd}
\end{equation*}

We now take the long exact sequence in cohomology associated to each
row to get the following commutative diagram:

\begin{equation*}
\begin{tikzcd}[column sep=5em,row sep=4em]
& 0
\arrow[r]
\arrow[d]
& A
\arrow[r,"1_{A}"]
\arrow[d,"\begin{pmatrix}1_{A}\\ 0   \end{pmatrix}"]
& A
\arrow[r]
\arrow[d,"1_{A}"]
& 0
\arrow[d]
& 
\\
0
\arrow[r]
& U
\arrow[r,"\begin{pmatrix}0\\ 1_{U}   \end{pmatrix} "]
& A\oplus U
\arrow[r,"B_{1}"]
\arrow[ur, phantom, "{\scriptstyle(1)}" description]
& A
\arrow[r]
& 0
& 
\\
0\arrow[r]
& H^{0}(E^{\vee})
\arrow[r]
\arrow[u]
\arrow[d]
& U^{\vee }
\arrow[ur, phantom, "{\scriptstyle(2)}" description]
\arrow[r,"B_{2}"]
\arrow[u,"\begin{pmatrix}\gamma \\ 0   \end{pmatrix}"]
\arrow[d,"\begin{pmatrix}0 \\ -j^{\vee }   \end{pmatrix}"]
& A^{\vee}
\arrow[r]
\arrow[u,"\phi^{-1}"]
\arrow[d,"1_{A^{\vee}}"]
& H^{1}(E^{\vee})
\arrow[r]
\arrow[u]
\arrow[d]
& 0
\\
0
\arrow[r]
& A^{\vee }
\arrow[r," \begin{pmatrix} -1_{A^{\vee}}\\ \alpha^{\vee }  \end{pmatrix}"]
& A^{\vee }\oplus A^{\vee }
\arrow[ur, phantom, "{\scriptstyle(3)}" description]
\arrow[r,"B_{3}"]
& A^{\vee }
\arrow[r]
& 0
&
\\
& 0
\arrow[r]
\arrow[u]
& A
\arrow[ur, phantom, "{\scriptstyle(4)}" description]
\arrow[r,"1_{A}"]
\arrow[u," \begin{pmatrix} 0\\ \phi   \end{pmatrix}"]
& A
\arrow[r]
\arrow[u,"\phi "]
& 0
\arrow[u]
&
\end{tikzcd}
\end{equation*}

By commutativity of the square (1) and exactness of row 2, we see that
$B_{1}$ must be $(1_{A},0)$. Similarly, we see from commutativity of
square (4) and exactness of the fourth row, $B_{3}$ must be
$(\alpha^{\vee},1_{A^{\vee}})$. Then we can compute $B_{2}$ in two
different ways: by the composition around the sides and bottom of
square (3) and by the composition around the sides and top of square
(2). The result is

\[
B_{2} = \phi\gamma  = -j^{\vee}.
\]
and the second equality is equivalent to the desired equality in the lemma.
\qed 

We are now in a position to complete the proof of Theorems~\ref{thm:
quiver description of Loop2Alg(LGr)} and \ref{thm: quiver description
of Loop2Alg(OGr)}. Using Proposition~\ref{prop: Loop2(Xpm) is a closed set in
Loop(Gr(n,2n)) satisfying condition involving theta and gamma} along
with the construction of $\phi$ and Equations~\eqref{eqns: closed
conditions} which were proved in Lemmas~\ref{lem: phivee = mp phi},
\ref{lem: phi.alpha=alphavee.phi}, and \ref{lem: phi.gamma = -jvee},
we see that 
\[
\LoopTwo (\Xpm ) = W^{\open}_{\mp,n,d} / GL(A)
\]
where $W_{\mp ,n,d}$ is the set of
\[
(\alpha ,\gamma ,j,\phi )\in \Hom (A,A)\oplus \Hom
(U^{\vee},A)\oplus  \Hom (A,U)\times \operatorname{Isom}(A,A^{\vee})
\]
where Equations~\eqref{eqns: closed conditions} hold and $W^{\open}_{\mp
,n,d}\subset W_{\mp ,n,d}$ is the open set given by the conditions

\bigskip

\begin{enumerate}
\item $\Ker (\alpha -\lambda \cdot \Id_{A})\cap \Ker (j) = 0$ for all
$\lambda$.
\item  $\im (\alpha -\lambda \cdot \Id_{A})+  \im (\gamma )=A$ for
all $\lambda$.
\end{enumerate}

\bigskip

In the symplectic case, all non-degenerate alternating forms on $A$
are isomorphic. In the orthogonal case, we restrict to split
non-degenerate symmetric forms, which form a single $GL(A)$-orbit.
Thus in either case we may fix such a form $\phi $ on $A$ and quotient
by its stabilizer
\[
G(A,\phi)=\{g\in GL(A):g^\vee \phi g=\phi\}.
\]

Since $\phi$ is invertible, we may eliminate $\gamma$ using the
equation $\gamma =-\phi^{-1} j^{\vee}$. Thus we get
\[
\LoopTwo (\Xpm ) = V_{\mp ,n,d}^{\open}/G(A,\phi )
\]
where $(A,\phi )$ is a $d$-dimensional vector space equipped with a 
non-degenerate form $\phi :A\to A^{\vee}$ satisfying $\phi^{\vee}=\mp
\phi$, assumed split in the symmetric case, 
\[
V_{\mp ,n,d} = \{(\alpha ,j)\in \Hom (A,A)\oplus \Hom (A,U),\quad \phi
\alpha =\alpha^{\vee}\phi \}, 
\]
and $V^{\open}_{\mp, n,d}\subset V_{\mp, n,d}$ is the open set satisfying

\bigskip

\begin{enumerate}
\item $\Ker (\alpha -\lambda \cdot \Id_{A})\cap \Ker (j) = 0$ for all
$\lambda$.
\item  $\im (\alpha -\lambda \cdot \Id_{A})+  \im (\phi^{-1}j^{\vee} )=A$ for
all $\lambda$.
\end{enumerate}
  
\bigskip

Condition (2) above is equivalent to condition (1). Indeed, we have
the following equivalences of conditions
\begin{align*}
\im (\alpha -\lambda \cdot \Id_{A})+  \im (\phi^{-1}j^{\vee} )&=A
\text{ for all } \lambda \iff \\
\im (\phi \alpha\phi^{-1} -\lambda \cdot \Id_{A^{\vee }})+  \im
(j^{\vee} )&=A^{\vee }
\text{ for all } \lambda \iff \\
\im (\alpha^{\vee } -\lambda \cdot \Id_{A^{\vee }})+ \im
(j^{\vee} )&=A^{\vee }
\text{ for all } \lambda \iff \\
\Ker (\alpha -\lambda \cdot \Id_{A})\cap \Ker (j) &= 0 \text{ for all } \lambda .
\end{align*}
The first equivalence uses the isomorphism $\phi$, the second
equivalence uses the equation $\phi \alpha  = \alpha^{\vee}\phi$, and
the final equivalence follows by duality.

Thus we have proved the following version of Theorems \ref{thm: quiver
description of Loop2Alg(LGr)} and \ref{thm: quiver description of
Loop2Alg(OGr)} with a more explicit description of the open condition:
\begin{proposition}\label{prop: quiver description of Loop2Alg(Xpm)
with explicit open condition}
Let $U$ be a vector space of dimension $n$ and let $(A,\phi )$ be a
$d$-dimensional vector space equipped with a non-degenerate form $\phi
:A\to A^{\vee}$ satisfying $\phi^{\vee}=\mp \phi$, where in the
symmetric case $\phi$ is assumed split. Let
\[
V_{\mp ,n,d} = \{(\alpha ,j)\in \Hom (A,A)\oplus \Hom (A,U):\quad \phi
\alpha =\alpha^{\vee}\phi \}
\]
and let $V^{\open}_{\mp ,n,d}\subset V_{\mp ,n,d}$ be the open set given
by the condition
\[
\Ker (\alpha -\lambda \cdot 1_{A})\cap \Ker (j) = 0\text{ for all } \lambda \in \overline{\kk}.
\]
Then there is a natural isomorphism
\[
\LoopTwo (\Xpm ) = V^{\open}_{\mp ,n,d}/G(A,\phi )
\]
where $G(A,\phi ) = \{g\in GL(A): g^{\vee}\phi g = \phi \}$. 
\end{proposition}

To complete the proof of Theorems \ref{thm: quiver description of
Loop2Alg(LGr)} and \ref{thm: quiver description of Loop2Alg(OGr)}, it
remains to show that $V^{\open}_{\mp ,n,d} = V^{\stable }_{\mp ,n,d}$
where $V^{\stable }_{\mp ,n,d}$ is the locus of GIT stable points for
the action of $G(A,\phi )$ on $V_{\mp ,n,d}$. To do this it will be
convenient to geometrically identify the full stack quotient $[V_{\mp
,n,d} / G(A,\phi )]$, i.e., to give it a useful moduli interpretation.

Revisiting the proof of Proposition~\ref{prop: quiver description of
Loop2Alg(Xpm) with explicit open condition} in the absence of the open
condition, we find that the stack $[V_{\mp ,n,d} / G(A,\phi )]$ is
equivalent to the stack of complexes on $\PP^{1}$ of the form
\[
\begin{tikzcd}[column sep={6em,between origins}]
E^{\vee } \arrow[r, "J_{\pm }^{-1}\circ \UL{f}^{\vee}"] &
\UL{U}\oplus \UL{U}^{\vee } \arrow[r, "\UL{f}"] & E,
\end{tikzcd}
\]
(``standard'' over $p_{\infty}$), exact in the middle, but where {\em $E$ is no
longer necessarily locally free and $\UL{f}$ is no longer necessarily
surjective}. Note that for self-dual sequences of the above form,
$\UL{f}$ is surjective if and only if $E$ is locally free (indeed,
this is the equivalence of conditions (1) and (2) proven above). It
follows that $E$ is locally free if and only if the sequence has no
non-trivial automorphisms. Moreover, the automorphism group of the
sequence is trivial if and only if it is finite. Thus the open set
$V^{\open}_{\mp ,n,d}$ is precisely the locus where $G(A,\phi )$ acts
freely (or equivalently with finite automorphisms).

Now a point $(\alpha ,j)\in V_{\mp ,n,d}$ is GIT stable for the
$G(A,\phi )$ action if and only if the orbit $G(A,\phi )\cdot (\alpha
,j)$ is closed and $(\alpha ,j)$ has finite stabilizer. Thus in
particular, $V^{\stable}_{\mp ,n,d}\subset V^{\open}_{\mp
,n,d}$. Conversely, if $(\alpha ,j)\in V^{\open}_{\mp ,n,d}$, then the
corresponding point $(\alpha ,j,-\phi^{-1}j^{\vee})\in V_{n,n,d}$ is
in $V^{\open}_{n,n,d}=V^{\stable}_{n,n,d}$ and hence the $GL(A)$ orbit
is a closed embedding of $GL(A)$ in $V_{n,n,d}$. Since $G(A,\phi
)\subset GL(A)$ is a closed subgroup, it follows that the orbit
$G(A,\phi )\cdot (\alpha ,j,-\phi^{-1}j^{\vee})$ is closed in
$V_{n,n,d}$ and hence $G(A,\phi )\cdot (\alpha ,j)$ is closed in
$V_{\mp ,n,d}$. Since we know points in $V^{\open}_{\mp ,n,d}$ have
finite stabilizers, it follows that $(\alpha ,j)$ is $G(A,\phi )$-stable and thus $V^{\open}_{\mp ,n,d}\subset V^{\stable }_{\mp ,n,d}$
which proves $V^{\open}_{\mp ,n,d}=V^{\stable }_{\mp ,n,d}$ and
completes the proofs of Theorems \ref{thm: quiver description of
Loop2Alg(LGr)} and \ref{thm: quiver description of
Loop2Alg(OGr)}. \qed

\section{Details of the topological argument}\label{sec: topological details}

In this section we provide the details for the topological
consequences of our work which we outlined in subsection~\ref{subsec:
topological consequences}. The results are mainly expressed as various
objects and isomorphisms in $\HoG$, the homotopy category of algebraic
spaces defined in \cite{Larson-Vakil-BottPeriodicity}. There is a
\emph{topological realization functor} 
\[
\HoG_{\CC } \xrightarrow{\top } \HoTop
\] 
from $\HoG_{\CC }\subset \HoG$, the subcategory of objects of $\HoG$ defined
over $\CC$,  to the homotopy category of topological spaces and so any
isomorphism in $\HoG_{\CC}$ yields a homotopy equivalence between the
corresponding topological spaces.

We list the results first and supply proofs afterward. We repeatedly (without comment) use the
standard homotopy equivalences between the general linear and unitary groups, since $GL$
is the natural language of the algebro-geometric results, and $U$ is the natural language
of the classical topological results.

\begin{proposition}\label{prop: Loop2(BU)=BGLd in HoG}
There is a sequence of inclusions
\[
\dotsb \hookrightarrow \LoopTwo (\Gr (n,n+N))\hookrightarrow \LoopTwo
(\Gr (n+1,n+N+2))\hookrightarrow \dotsb
\]
which gives a well-defined object in $\HoG$ which we denote $\LoopTwo
(BU)$. Moreover, there is an isomorphism
\[
\LoopTwo (BU)\cong BGL_{d}
\]
in $\HoG$. 
\end{proposition}

Taking topological realizations yields

\begin{corollary}
The inductive limit of the varieties $\LoopTwo (\Gr (n,n+N))$ under
the above inclusions, in the complex analytic topology, has the
homotopy type of $BU(d)$:
\[
\lim_{n,N\to \infty} \LoopTwo (\Gr (n,n+N)) \homotopyeq BU(d).
\]
\end{corollary}

\begin{proposition}\label{prop: Loop2inft(BU)=BGL in HoG}
There is a sequence of inclusions 
\[
\dotsb \hookrightarrow \LoopTwo (\Gr (n,n+N))
\hookrightarrow \Omega^{2}_{d+1,\alg} (\Gr (n+1,n+N+2)) \hookrightarrow \dotsb 
\]
which gives a well-defined object in $\HoG$ which we denote
$\Omega^{2}_{\infty ,\alg}(BU)$. Moreover, there is an isomorphism
\[
\Omega^{2}_{\infty ,\alg}(BU)\cong BGL
\]
in $\HoG$. 
\end{proposition}

The Boyer-Mann-Hurtubise-Milgram theorem
(Theorem~\ref{thm: Mann-Milgram algebraic loop approximation
holds for generalized flag varieties})
says that in the limit as $d\to \infty$, the
algebraic loop spaces of the Grassmannians are homotopy equivalent to
their topological loop spaces. Then since
\[
\lim_{n,N\to \infty} \Gr (n,n+N) \homotopyeq  BU,
\]
we have
\[
\lim_{d,n,N\to \infty} \LoopTwo (\Gr (n,n+N)) \homotopyeq \LoopTwoTop (BU), 
\]
where on the right-hand side $d$ is any fixed value (they are all
homotopy equivalent). 

Consequently, after taking the topological realization of the
isomorphism in the proposition and composing with the above homotopy
equivalence, we recover the complex Bott periodicity homotopy
equivalence:

\begin{corollary}
\[
\LoopTwoTop (BU)\homotopyeq BU.
\]
\end{corollary}

We have analogous results for the algebraic loop spaces of the
isotropic orthogonal Grassmannian and the Lagrangian
Grassmannian. Recall the notation
\[
X_{+,n} =  O(2n)/U(n) = \OGr (n,2n),\quad  \quad X_{-,n} = Sp(2n)/U(n)=\LGr (n,2n).
\]

\begin{proposition}\label{prop: Loop2(X+,inrty)=BSp(d) in HoG and same
for X-}
There is a sequence of inclusions
\[
\dotsb \hookrightarrow \LoopTwo (X_{\pm ,n})  \hookrightarrow \LoopTwo
(X_{\pm ,n+1}) \hookrightarrow \dotsb 
\]
which gives well-defined objects in $\HoG$ which we denote $\LoopTwo
(X_{\pm,\infty })$. Moreover, there are isomorphisms
\begin{align*}
\LoopTwo (X_{+,\infty})&\cong BSp(d)\\
\LoopTwo (X_{-,\infty})&\cong BO(d)
\end{align*}
in $\HoG$. 
\end{proposition}
Taking topological realizations yields:
\begin{corollary}
The inductive limits of the varieties $\LoopTwo (X_{\pm ,n})$ under
the above inclusions, in the complex analytic topology, have the
homotopy types:
\begin{align*}
\lim_{n\to \infty }\LoopTwo (X_{+,n})&\homotopyeq  BSp(d)\\
\lim_{n\to \infty }\LoopTwo (X_{-,n})&\homotopyeq  BO(d).
\end{align*}
\end{corollary}

Finally, we have
\begin{proposition}\label{prop: Loop2inftyX+=BSp in HoG and similar for X-}
There is a sequence of inclusions
\[
\dotsb \hookrightarrow \LoopTwo (X_{\pm ,n})
\hookrightarrow \Omega^{2}_{d+2,\alg} (X_{\pm ,n+2}) \hookrightarrow \dotsb 
\]
which give well-defined objects in $\HoG$ which we denote by
$\Omega^{2}_{\infty ,\alg}(X_{\pm ,\infty})$. Moreover, there are
isomorphisms
\begin{align*}
\Omega^{2}_{\infty ,\alg}(X_{+,\infty})&\cong BSp\\
\Omega^{2}_{\infty ,\alg}(X_{-,\infty})&\cong BO
\end{align*}
in $\HoG$.
\end{proposition}

Again by Theorem~\ref{thm: Mann-Milgram algebraic loop approximation
holds for generalized flag varieties} in the limit as $d\to \infty$
the algebraic double loop spaces of $X_{\pm,n}$ are homotopy
equivalent to the topological double loop spaces. Then since 
\[
\lim_{n\to \infty}X_{+,n} = O/U, \quad \lim_{n\to \infty}X_{-,n} = Sp/U, 
\]
we have
\begin{align*}
\lim_{d,n\to \infty} \LoopTwo (X_{+,n}) & \homotopyeq \LoopTwoTop (O/U),\\
\lim_{d,n\to \infty} \LoopTwo (X_{-,n}) & \homotopyeq \LoopTwoTop (Sp/U)
\end{align*}
where on the right-hand side, $d$ is any fixed value (they are all
homotopy equivalent).

Consequently, after taking the topological realizations of the
isomorphisms in the proposition, we recover two of the real Bott
periodicity homotopy equivalences:
\begin{corollary}
\begin{align*}
\LoopTwoTop (O/U)& \homotopyeq BSp,\\
\LoopTwoTop (Sp/U)& \homotopyeq BO.
\end{align*}
\end{corollary}

\subsection{Proof of Proposition~\ref{prop: Loop2(BU)=BGLd in HoG}}

The inclusions 
of the proposition arise from applying $\LoopTwo (-)$ to
the natural inclusion
\[
\Gr (n,n+N)\hookrightarrow \Gr (n+1,n+N+2). 
\]
On the level of bundles on $\PP^{1}$, this corresponds to 
\[
\left\{\UL{U}\oplus \UL{W} \xrightarrow{(f_{U},f_{W})}E
\right\}\mapsto \left\{\UL{U}\oplus \OO \oplus \UL{W}\oplus \OO
\xrightarrow{(f_{U}\oplus \Id ,f_{W}\oplus 0)}E\oplus \OO
\right\}
\]
and on the level of quivers is given by
\[
(A,U,W,(\alpha ,j,\gamma ))\mapsto (A,U\oplus \kk ,W\oplus \kk ,(\alpha ,(j,0),(\gamma ,0))).
\]
We thus get a diagram

\begin{equation} \label{eqn: HoG diagram for Loop2(d,BU)}
\begin{tikzcd}
    \cdots \arrow[r, hook]
    & V^{\stable}_{n,N,d}/GL_{d} \arrow[r, hook] \arrow[d, hook]
    & V^{\stable}_{n+1,N+1,d}/GL_{d}  \arrow[r, hook] \arrow[d, hook] & \cdots \\
    \cdots \arrow[r, hook]
    & \left[V_{n,N,d}/GL_{d} \right] \arrow[r, hook]
    & \left[V_{n+1,N+1,d}/GL_{d} \right] 
\arrow[r, hook] & \cdots
\end{tikzcd}
\end{equation}
where by Theorem~\ref{thm: quiver description of Loop2alg(Gr(n,n+N))},
the top row is the sequence in the proposition and the bottom row is
the inclusion of smooth stack quotients. By Larson-Vakil
\cite[Definition 3.1]{Larson-Vakil-BottPeriodicity}, this diagram
gives an object in $\HoG$ provided that we can show the vertical open
inclusions are of increasing connectivity (the other requirements are
immediate). This follows from the following
\begin{lemma}\label{lem: codimension of unstable locus for the Gr quiver}
The complement of the open embedding
$V^{\stable}_{n,N,d}\hookrightarrow V_{n,N,d}$ has codimension $\min
(n,N)$ and thus the embedding is $\min(n,N)-1$ \emph{connected} (in
the sense of \cite[Definition~2.1]{Larson-Vakil-BottPeriodicity}).
\end{lemma}
\begin{proof}
We wish to show that the unstable locus has codimension
$\min(n,N)$. The unstable locus has two components given by the
negation of conditions (1) and (2) in Proposition~\ref{prop: quiver
description of Loop2(Gr(n,n+N)) with explicit stability condition}.

Recall that geometrically, the stack $\left[V_{n,N,d}/GL_{d} \right]$
parameterizes sequences
\[
\left\{\UL{U}\oplus \UL{W}\to E \right\}.
\]
Failure of condition (1) corresponds to the locus where the sheaf $E$
is not locally free. Failure of condition (2) corresponds to the locus
where the map is not surjective. These two conditions are exchanged
under dualizing the quiver, or geometrically, taking the (derived)
dual of the sequence 
\[
F\to \UL{U}\oplus \UL{W}\to E.
\]
Recall that 
\[
\dim (V^{\stable}_{n,N,d}/GL_{d}) = \dim (\left[V_{n,N,d}/GL_{d}
\right]) =  (n+N)d.
\]
We compute the dimension of the substack parameterizing the locus of
sequences 
\[
\left\{\UL{U}\oplus \UL{W}\xrightarrow{f}E \right\}
\]
where $f$ fails to be surjective. Since $f|_{p_{\infty}}$ is
surjective, the cokernel of $f$ is a 0-dimensional sheaf and so
generically, this locus parameterizes such maps where $E$ is locally
free and $\coker f \cong \OO_{p}$ for some $p\in \PP^{1}-p_{\infty}$.

This generic locus admits a map to
$V^{\stable}_{n,N,d-1}/GL_{d-1}$ given by 
\[
\left\{\UL{U}\oplus \UL{W}\xrightarrow{f}E \right\}\mapsto
\left\{\UL{U}\oplus \UL{W}\xrightarrow{f' }E' \right\}
\]
where $E'\subset E$ is the image of $f$. The fibers of this map are
stacks parameterizing extensions
\[
0\to E' \to E\to \OO_{p}\to 0.
\]
These fibers thus have dimension
\[
\dim (\PP (\Ext^{1}(\OO_{p},E'))) +\dim  (\PP^{1}-p_{\infty}) =
(n-1)+1 = n.
\]
Therefore the dimension of the locus where condition (2) fails is 
\[
n+(d-1)(N+n)
\]
and so its codimension in $\left[V_{n,N,d}/GL_{d} \right]$ is 
\[
d(N+n) - n -(d-1)(N+n) = N.
\]
Then by duality, the locus where condition (1) fails has codimension
$n$. The lemma follows.
\end{proof}

We have thus proven that $\LoopTwo (BU)$ is a well-defined object in
$\HoG$. The isomorphism $\LoopTwo (BU)\cong BGL_{d}$ follows
immediately since the bottom row of Equation~\eqref{eqn: HoG diagram
for Loop2(d,BU)} is isomorphic to $BGL_{d}$ via the affine bundle maps
\[
\left[V_{n,N,d}/GL_{d} \right] \to BGL_{d}.
\]

\subsection{Proof of Proposition~\ref{prop: Loop2inft(BU)=BGL in HoG}}

The inclusions 
\[
\LoopTwo (\Gr (n,n+N))\hookrightarrow
\Omega^{2}_{d+1,\alg}(\Gr (n+1,n+N+2))
\]
of Proposition~\ref{prop:
Loop2inft(BU)=BGL in HoG} are defined geometrically by

\[
\left\{\UL{U}\oplus \UL{W}\xrightarrow{(f_{U},f_{W})}E \right\}\mapsto
\left\{\UL{U}\oplus \OO \oplus \UL{W}\oplus \OO
\xrightarrow{\begin{psmallmatrix}f_{U}&0&f_{W}&0\\ 0&\UL{x}_{0}&0& \UL{x}_{1}   \end{psmallmatrix}
}E\oplus \OO (1) \right\}
\]
which in terms of quivers is given by
\[
(A,U,W,(\alpha ,j,\gamma))\mapsto (A,U,W,(\alpha ,j, \gamma ))\oplus
(\kk ,\kk ,\kk ,(0,1,1)). 
\]
Thus we get a diagram
\[
\begin{tikzcd}
    \cdots \arrow[r, hook]
    & V^{\stable}_{n,N,d}/GL_{d} \arrow[r, hook] \arrow[d, hook]
    & V^{\stable}_{n+1,N+1,d+1}/GL_{d+1}  \arrow[r, hook] \arrow[d, hook] & \cdots \\
    \cdots \arrow[r]
    & \left[V_{n,N,d}/GL_{d} \right] \arrow[r]
    & \left[V_{n+1,N+1,d+1}/GL_{d+1} \right] 
\arrow[r] & \cdots
\end{tikzcd}
\]
To show this gives a well-defined object in $\HoG$ we need to show the
vertical open inclusions have increasing connectivity and also that
the inclusions on the bottom row have increasing connectivity. The
former follows from Lemma~\ref{lem: codimension of unstable locus for
the Gr quiver} and the latter follows from
\cite[Prop~2.2]{Larson-Vakil-BottPeriodicity} and the fact that
$\left[V_{n,N,d}/GL_{d} \right] \to BGL_{d}$ is
$\infty$-connected.

This then shows that the object $\Omega^{2}_{\infty ,\alg}(BU)$ is
isomorphic to $BGL$ in $\HoG$ since the bottom row of the diagram may
be replaced with 
\[
\dotsb \hookrightarrow BGL_{d}\hookrightarrow BGL_{d+1}\hookrightarrow \dotsb .
\]

\subsection{Proof of Proposition~\ref{prop: Loop2(X+,inrty)=BSp(d) in HoG and same
for X-}}

The inclusions in the proposition arise from applying $\LoopTwo (-)$
to the natural inclusions $X_{\mp ,n}\hookrightarrow X_{\mp ,n+1}$. In
terms of sequences on $\PP^{1}$, this corresponds to 
\[
\left\{\UL{U}\oplus \UL{U}^{\vee}\xrightarrow{(f_{U},f_{U^{\vee}})}E
\right\} \mapsto \left\{\UL{U}\oplus \OO \oplus \UL{U}^{\vee}\oplus
\OO  \xrightarrow{(f_{U},\Id ,f_{U^{\vee}},0)}E\oplus \OO  \right\}.
\]
In terms of quivers, this is given by
\[
\left\{(A,\phi ),U,(\alpha ,j) \right\}\mapsto \left\{(A,\phi
),U\oplus \kk ,(\alpha ,(j,0)) \right\}.
\]
Thus we get a diagram
\begin{equation}\label{diag: diagram defining Loop(X+infty) in HoG}
\begin{tikzcd}
    \cdots \arrow[r, hook]
    & V^{\stable}_{\mp ,n,d}/G(A,\phi ) \arrow[r, hook] \arrow[d, hook]
    & V^{\stable}_{\mp ,n+1,d}/G(A,\phi )  \arrow[r, hook] \arrow[d, hook] & \cdots \\
    \cdots \arrow[r, hook]
    & \left[V_{\mp ,n,d}/G(A,\phi ) \right] \arrow[r, hook]
    & \left[V_{\mp, n+1,d}/G(A,\phi ) \right] 
\arrow[r, hook] & \cdots
\end{tikzcd}
\end{equation}
As in the proof of Proposition~\ref{prop: Loop2(BU)=BGLd in HoG}, we
need only show the vertical inclusions are open embeddings of
increasing connectivity. This follows from the following

\begin{lemma}\label{lem: the complement of Vstable in V in the Sp/O
case has codim n} The complement of the open embedding $V^{\stable}_{+
,n,d}\subset  V_{+ ,n,d}$ has codimension $n$ and the
complement of the open embedding $V^{\stable}_{- ,n,d}\subset 
V_{- ,n,d}$ has codimension $n-1$.
Consequently, the open embeddings  $V^{\stable}_{\mp  ,n,d}\hookrightarrow
V_{\mp  ,n,d}$ are $(n-2)$-connected. 
\end{lemma}
\begin{proof}
We denote the complement of $V^{\stable}_{\mp ,n,d}\subset V_{\mp
,n,d}$ by $Z$. It is given by the negation of the open condition in
Proposition~\ref{prop: quiver description of Loop2Alg(Xpm) with
explicit open condition}:
\[
Z=\left\{(\alpha ,j)\in V_{\mp ,n,d}: \exists \lambda \in \overline{\kk} ,v\in A\text{ such that
}j(v)=0\text{ and } \alpha v=\lambda v \right\}.
\]
Let us abbreviate $V_{\mp ,n,d}$ to $V_{\mp}$ and we 
identify it as follows:
\begin{align*}
V_{+}&=\Sym^{2}A^{\vee}\oplus \Hom (A,U),\\
V_{-}&=\Lambda ^{2}A^{\vee}\oplus \Hom (A,U).
\end{align*}
To make the identification, we note that $\alpha \in \Hom (A,A)$ is
uniquely determined by the associated bilinear form $B_{\alpha}$ given
by 
\[
B_{\alpha}(v,w) = \left\langle v,\alpha w \right\rangle
\]
where $\left\langle -,- \right\rangle$ is the non-degenerate
orthogonal/symplectic form associated to $\phi :A\to A^{\vee}$. The
condition $\alpha^{\vee } \phi =\phi \alpha$ means that $\alpha$ is
self-adjoint: 
\[
\left\langle v,\alpha w \right\rangle=\left\langle \alpha v,w \right\rangle
\]
which implies the associated form $B_{\alpha}$ is also
symmetric/skew. Under this identification, the condition $\alpha
v=\lambda v$ translates to an equality in $A^{\vee}$:
\[
B_{\alpha}(-,v) = \lambda \left\langle -,v \right\rangle
\]
so $Z$ may be rewritten as
\[
Z=\left\{(B_{\alpha } ,j)\in V_{\mp}: \exists \lambda \in \overline{\kk} ,v\in A\text{ such that
}j(v)=0\text{ and }  B_{\alpha}(-,v) = \lambda \left\langle -,v \right\rangle\right\}.
\]

To compute the dimension of $Z$ we first compute the dimension of the
incidence variety
\[
\tilde{Z} =\left\{([v],(B_{\alpha},j),\lambda )\in \PP (A)\times
V_{\mp}\times \AA^{1} \quad :\quad jv=0, B_{\alpha}(-,v) = \lambda
\left\langle -,v \right\rangle   \right\}.
\]
We see that $\tilde{Z}\to \PP (A)$ is a subbundle of the trivial
bundle $\PP (A)\times V_{\mp}\times \AA^{1}$ since the conditions are
linear for fixed $[v]$.

The condition $jv=0$ imposes $n$ linear conditions on $V_{\mp}$,
however the condition $B_{\alpha}(-,v) = \lambda \left\langle -,v
\right\rangle$ imposes a different number of conditions in the two
cases $V_{\mp}$. Let 
\[
L_{v}: A^{\vee}\otimes A^{\vee}\oplus \AA^{1} \to A^{\vee}
\]
denote the map given by 
\[
(B_{\alpha}(-,-),\lambda )\mapsto B_{\alpha}(-,v)-\lambda \left\langle
-,v \right\rangle .
\]
Since $B_{\alpha}$ lies in $\Sym^{2}A^{\vee}$ in the $V_{+}$ case and
$\Lambda^{2}A^{\vee}$ in the $V_{-}$ case, we consider each case
separately: 
\begin{align*}
\Sym^{2}A^{\vee} \oplus \AA^{1} &\xrightarrow{L_{v}}A^{\vee}\\
\Lambda ^{2}A^{\vee} \oplus \AA^{1} &\xrightarrow{L_{v}}A^{\vee}.
\end{align*}
The former case is surjective, but the latter has image given by the
kernel of the map
\begin{align*}
A^{\vee}&\xrightarrow{ev_{v}} \AA^{1} \\
f(-) &\mapsto f(v)
\end{align*}
since 
\[
B_{\alpha}(v,v)=\left\langle v,v \right\rangle = 0
\]
for skew forms. Therefore the condition 
\[
L_{v}(B_{\alpha},\lambda )=0
\]
imposes $d$ conditions in the $V_{+}$ case and $d-1$ conditions in the
$V_{-}$ case. 

Returning to $\tilde{Z}$, we see it is a subbundle of corank $n+d$ in
the $V_{+}$ case and $n+d-1$ in the $V_{-}$ case. Then since $Z\subset
V_{\mp}$ is the image of $\tilde{Z}$ under the projection 
\[
\PP (A)\times V_{\mp}\times \AA^{1} \to V_{\mp }
\]
we see that $Z$ has codimension given by 
\[
\operatorname{codim}(\tilde{Z})-d = \begin{cases}
n&\text{for $V_{+}$}\\
n-1&\text{for $V_{-}$}
\end{cases}
\]
which proves the lemma.
\end{proof}
Now that we have proved that $\LoopTwo (X_{\pm,\infty })$ is
well-defined in $\HoG$ via Diagram~\eqref{diag: diagram defining
Loop(X+infty) in HoG}, it remains to prove the isomorphisms in
Proposition~\ref{prop: Loop2(X+,inrty)=BSp(d) in HoG and same for
X-}. This
then follows from observing that $G(A,\phi )\cong Sp(d)$ in the $X_{+}$
case and $G(A,\phi )\cong O(d)$ in the $X_{-}$ case and that the
bottom row of Diagram~\eqref{diag: diagram defining Loop(X+infty) in
HoG} maps to $BG(A,\phi )$ with affine fibers and thus defines an
isomorphism in $\HoG$.

\subsection{Proof of Proposition~\ref{prop: Loop2inftyX+=BSp in HoG and similar for X-}}
  
The inclusion
\[
\LoopTwo (X_{\pm ,n})\hookrightarrow \Omega^{2}_{d+2,\alg}(X_{\pm ,n+2})
\]
is given in terms of quivers by
\[
\left\{(A,\phi ),U,(\alpha ,j) \right\} \mapsto \left\{(A\oplus
\AA^{2},\phi\oplus \phi_{0} ),U\oplus \AA^{2},(\alpha\oplus \alpha_{0}
,j\oplus j_{0}) \right\} 
\]
where
\[
\phi_{0} = \begin{pmatrix} 0&1\\\mp 1&0  \end{pmatrix}\quad
\alpha_{0}=\begin{pmatrix} 0&0\\0&0  \end{pmatrix}\quad
j_{0}=\begin{pmatrix} 1&0\\0&1  \end{pmatrix}. 
\]
In terms of sequences of sheaves, this is given by 
\[
\left\{\UL{U}\oplus \UL{U}^{\vee}\xrightarrow{(f_{U},f_{U^{\vee}})} E
\right\} \mapsto
\left\{\UL{U}\oplus \OO^{2}\oplus \UL{U}^{\vee}\oplus
\OO^{2}\xrightarrow{\begin{psmallmatrix} f_{U}&0&0&f_{U^{\vee}} &0&0\\
0&x_{0}&0&0&0&\pm x_{1}\\
0&0&x_{0}&0&-x_{1}&0  \end{psmallmatrix}} E\oplus \OO (1)\oplus \OO (1) \right\}.
\]
This gives us a diagram
\begin{equation}\label{digram: Vnd/Gd --> Vn+2d+2/Gd+2}
\begin{tikzcd}
    \cdots \arrow[r, hook]
    & V^{\stable}_{\mp ,n,d}/G_{\mp }(d) \arrow[r, hook] \arrow[d, hook]
    & V^{\stable}_{\mp ,n+2,d+2}/G_{\mp }(d+2)  \arrow[r, hook] \arrow[d, hook] & \cdots \\
    \cdots \arrow[r]
    & \left[ V_{\mp ,n,d}/G_{\mp }(d)\right] \arrow[r]
    & \left[V_{\mp ,n+2,d+2}/G_{\mp }(d+2) \right] 
\arrow[r] & \cdots
\end{tikzcd}
\end{equation}
where we have denoted $G(A,\phi )$ by $G_{\mp}(d)$ to emphasize the
dependence on $d=\dim A$ and the sign of the form $\phi^{\vee}=\mp
\phi$.

The diagram gives a well-defined object in $\HoG$ provided that we can
prove the vertical open inclusions are of increasing connectivity and
also that the morphisms on the bottom row have increasing
connectivity. The former follows from Lemma~\ref{lem: the complement
of Vstable in V in the Sp/O case has codim n} and the latter follows
from the following:
\begin{lemma}
The morphism $BG_{\mp }(d)\to BG_{\mp }(d+2)$ is $d$-connected.
\end{lemma}
\begin{proof}
We treat the skew and symmetric cases separately. First suppose that
$\left\langle -,- \right\rangle$ is skew, so that $G_{-}$ is symplectic.
Let
\[
Y_{-}=\left\{(e,f)\in A^{2}: \quad \left\langle e,f \right\rangle = 1 \right\}.
\]
There is a surjective map $Y_{-}\to A-\{0 \}$ given by $(e,f)\mapsto e$
whose fibers are affine linear hyperplanes in $A$ (since $\left\langle
-,- \right\rangle$ is non-degenerate). The group $G_{-}(d+2)$ acts
transitively on $Y_{-}$, and the stabilizer of a pair $(e,f)$ is the
symplectic group of the orthogonal complement of the span of $e$ and
$f$, which is isomorphic to $G_{-}(d)$. Thus
\[
BG_{-}(d)\cong [Y_{-}/G_{-}(d+2)].
\]
We therefore have a sequence of maps
\begin{align*}
BG_{-}(d)\xrightarrow{\quad \cong\quad  }&
\left[Y_{-}/G_{-}(d+2)\right]\\
\xrightarrow{\text{affine bundle}}&
\left[A-\{0 \} / G_{-}(d+2) \right]\\
\xhookrightarrow{\text{$(d+1)$-connected}}&
\left[A / G_{-}(d+2) \right]\\
\xrightarrow{\text{affine bundle}}&
BG_{-}(d+2).
\end{align*}
It follows from the two-out-of-three rule that $BG_{-}(d)\to
BG_{-}(d+2)$ is $d$-connected.

Now suppose that $\left\langle -,- \right\rangle$ is symmetric, so that
$G_{+}$ is orthogonal. In this case the correct analogue of $Y_{-}$ is
the space of hyperbolic pairs
\[
Y_{+}=\left\{(e,f)\in A^{2}: \quad
\left\langle e,e \right\rangle=0,\quad
\left\langle f,f \right\rangle=0,\quad
\left\langle e,f \right\rangle=1 \right\}.
\]
By Witt extension, $G_{+}(d+2)$ acts transitively on $Y_{+}$, and the
stabilizer of a hyperbolic pair $(e,f)$ is the orthogonal group of the
orthogonal complement of the hyperbolic plane spanned by $e$ and $f$.
This stabilizer is isomorphic to $G_{+}(d)$. Hence
\[
BG_{+}(d)\cong [Y_{+}/G_{+}(d+2)].
\]
The morphism
\[
[Y_{+}/G_{+}(d+2)]\to BG_{+}(d+2)
\]
is the standard orthogonal Stiefel stabilization morphism obtained by
adjoining a hyperbolic plane, and is $d$-connected. Equivalently, this
is the usual $d$-connected stabilization map $BO(d)\to BO(d+2)$ modeled
by the homogeneous space of hyperbolic pairs $G_{+}(d+2)/G_{+}(d)$.
Thus $BG_{+}(d)\to BG_{+}(d+2)$ is $d$-connected as well.
\end{proof}

Thus diagram~\eqref{digram: Vnd/Gd --> Vn+2d+2/Gd+2} gives
well-defined objects $\Omega^{2}_{\infty ,\alg }(X_{\pm 
,\infty})$ in $\HoG$ and from the bottom row of the diagram we get the
isomorphisms
\[
\Omega^{2}_{\infty ,\alg }(X_{\pm,\infty}) \cong BG_{\mp} 
\]
where $BG_{\mp}$ is the object in $\HoG$ given by the sequence
\begin{equation}\label{eqn: sequence defining BGmp in HoG}
\dotsb \to BG_{\mp}(d)\to BG_{\mp}(d+2) \to \dotsb 
\end{equation}
and the isomorphisms are via the maps $\left[V_{\mp ,n,d}/G_{\mp}(d)
\right]\to BG_{\mp}(d)$ which have affine fibers. These isomorphisms
are the desired isomorphism from Proposition~\ref{prop:
Loop2inftyX+=BSp in HoG and similar for X-} since
\[
BG_{+}=BO \quad  BG_{-}=BSp.
\]

We remark that although the group scheme $O(A,\phi)$ depends on the
isometry class of the quadratic form $\phi$, the unpointed
classifying stack $BO(A,\phi)$ is, up to non-canonical equivalence,
determined only by the dimension of $A$.\footnote{If $\phi$ and $\psi$ are two
non-degenerate quadratic forms of the same rank, then the sheaf
\[
\Isom((A,\phi),(A,\psi))
\]
is an $O(A,\psi)$-$O(A,\phi)$-bitorsor, and hence induces an
equivalence
\[
BO(A,\phi)\simeq BO(A,\psi).
\]
Equivalently, $O(A,\phi)$-torsors classify non-degenerate quadratic
bundles of rank $\dim A$.}
Moreover, even though our representation of $BO$ in $\HoG$ is given
by the sequence~\eqref{eqn: sequence defining BGmp in HoG}, which
increases dimension by $2$, one can also construct highly connected
maps $BO(n)\to BO(n+1)$.
More precisely:  we have two objects $BO_{\text{odd}}$ and $BO_{\text{even}}$, elements of $\HoG$,
given by $BO(1) \rightarrow BO(3) \rightarrow BO(5) \rightarrow \cdots$ and
$BO(2) \rightarrow BO(4) \rightarrow BO(6) \rightarrow \cdots$ respectively, and inverse
morphisms between then (in $\HoG$) given by
$$
\xymatrix{
  BO(1) \ar[r] \ar[d] & BO(3) \ar[r] \ar[d] & BO(5) \ar[r] \ar[d] & \cdots \\
  BO(2) \ar[r]  & BO(4) \ar[r]  & BO(6) \ar[r]  & \cdots}
$$
and
$$
\xymatrix{
  BO(2) \ar[r] \ar[d] & BO(4) \ar[r] \ar[d] & BO(6) \ar[r] \ar[d] & \cdots \\
  BO(3) \ar[r]  & BO(5) \ar[r]  & BO(7) \ar[r]  & \cdots}
$$
inducing an isomorphism $BO_{\text{odd}} \overset \sim \longleftrightarrow BO_{\text{even}}$.
Thus $BO$ is well-defined up to isomorphism
in $\HoG$.

This completes the proof of Proposition~\ref{prop: Loop2inftyX+=BSp in
HoG and similar for X-}. \qed

\bibliography{../BV-biblio}
\bibliographystyle{plain}

\end{document}